\documentclass{article}
\usepackage[latin1]{inputenc}
\usepackage{a4wide}
\usepackage{amsmath}
\usepackage{amssymb}
\usepackage{amsfonts,euler}
\usepackage{amsthm}
\usepackage{url}
\usepackage{graphicx}
\usepackage{nicefrac}
\usepackage{prettyref}
\usepackage{color}
\usepackage{enumitem}
\usepackage{dsfont}
\usepackage[linktocpage=true]{hyperref}

\usepackage[normalem]{ulem}

\newcommand{\E}{\ensuremath{\mathbb{E}}}

\newcommand{\Prob}{\ensuremath{\mathbb{P}}}

\newtheorem{thm}{Theorem}[section]
\newtheorem{lem}{Lemma}[section]
\newtheorem{cor}[thm]{Corollary}
\newtheorem{rem}{Remark}
\newtheorem{proposition}[thm]{Proposition}

\newcommand{\Rset}{\ensuremath{\mathbb{R}}}
\newcommand{\Nset}{\ensuremath{\mathbb{N}}}
\newcommand{\bo}{\ensuremath{\mathrm{O}}}
\newcommand{\Cset}{\ensuremath{\mathbb{C}}}
\newcommand{\Zset}{\ensuremath{\mathbb{Z}}}

\title{Central limit theorem analogues for multicolour urn
  models}
\author{Noela M\"{u}ller \\
Institute of Mathematics \\
J.W.~Goethe University \\
60054 Frankfurt am Main\\
Germany\\ \\
Email: {\tt nmueller@math.uni-frankfurt.de}}

\begin{document}
\maketitle

\begin{abstract}
The asymptotic behaviour of a generalised
P\'olya--Eggenberger urn is well--known to depend on the spectrum of its replacement matrix: If its dominant eigenvalue
$r$ is simple and no other eigenvalue is ``large'' in the sense that its real
part is greater than $r/2$, the
normalized urn composition is asymptotically normally distributed. However, if
there is more than one large eigenvalue,  the first few random draws have a non--negligible effect on the evolution of the urn process and almost sure random
tendencies of order larger than $\sqrt{n}$ 
typically prevent a
classical central limit theorem. In the present work, a
central limit theorem analogue for the fluctuations of urn models with
regard to random linear drift and random periodic growth of order
larger than $\sqrt{n}$ is proved, covering the $m$-ary search tree and B-trees. 
The proof builds on an eigenspace decomposition of the process in order to separate components of different growth orders. By an accurately tailored adaption of martingale
techniques to the components, their joint limiting behaviour is established and translated back to the urn process. Conveniently, the approach encompasses results on small urn
models and therefore provides a unifying perspective on central
limit theorems for certain urn models, irrespective of their spectrum.
\end{abstract}

\noindent
{\textbf{MSC2010:}} 60F05, 60F15, 60C05, 60J10.

\noindent
\textbf{Keywords:} Generalised P\'olya urn, cyclic urn, m-ary search tree, weak
convergence, central limit theorem analogue, martingale central limit theorem.

\section{Introduction}
Urn models are prominent probabilistic schemes with applications in a
variety of areas. Their popularity is owed both to their simplicity as
to their flexibility. In the present article, we focus on the
following model: Consider an urn that evolves in discrete time and
contains balls of $q \in \Nset_{\geq 2}$ different colours, labelled
$1, \ldots, q$. Its stochastic evolution is determined by two
deterministic parameters, namely the ball configuration at time $0$
and a scheme $R \in \Zset^{q \times q}$ which represents the transition
rules. The matrix $R$ is called \textit{generating matrix} of the following picturesque process: 
Immediately before each time $n \in \mathbb{N}$, a ball is drawn from the urn uniformly at
random and independently of the previous draws. If the drawn ball has colour $j$, it is put back into the urn together with $R_{ij}$ balls of colour $i$, for each $i \in \{1, \ldots, q\}$. If $R_{ij}<0$, $|R_{ij}|$ balls of colour $i$ are removed from the urn. 
In the course of the paper, we
assume that $R$ is diagonisable over $\mathbb{C}$ and that the same positive number of balls is added at each
step, such that the number of balls in the urn at time $n$ is deterministic. For a precise formulation of all assumptions see (A$1$) - (A$4$) in the beginning of section \ref{Main}. We denote the number of balls of type $i$ after $n$ draws from the urn by $X_n^{(i)}$. The vector $X_n:=(X_n^{(1)}, \ldots, X_n^{(q)})^t$ is called the \textit{urn composition} at time $n$. 

Important standard references on urn models are \cite{J,M}. The probably most exhaustive treatment of a general class of urn models is \cite{J}, 
where further references on the subject can be found.
The importance of the spectrum of $R$ for the asymptotic behaviour of the urn
composition has been known for a long time
\cite{athreya1968embedding, KeSti}. In the current setting, the asymptotics of $(X_n)_{n \geq 0}$ are as follows: The
proportions 
of balls of the different types converge almost
surely as the number of draws tends to infinity \cite{Gou}. Whether the limit $V$ is random or deterministic depends on the initial
composition of the urn and the multiplicity of the largest
eigenvalue $r$ of $R$. The non--zero components of $V$ correspond to so--called \textit{dominant} colours, which may be structured in classes, see section \ref{Main}.
If the multiplicity of $r$ is exactly one, the rescaled urn
composition vector is asymptotically
normally distributed \cite{athreya1968embedding, J, Pou,S, LP}. On the other hand, if there is more than one
eigenvalue with real part greater than $r/2$, the asymptotic expansion of the urn composition typically contains random terms of size larger than $\sqrt{n}$ which may even oscillate, compare \cite{CP, ChHw, Pou}. In such a situation, the fluctuation about these random tendencies is of some interest \cite{KuSu, MN}. This question is addressed in the present article, and in order to study the fluctuations, we employ a ``non-classical'' normalisation of the urn composition vector that involves random
centering and possibly random scaling. Following
\cite{HH}, we use the term
central limit theorem {\em analogue} to refer to
this type of result.  Our main result is the following theorem.



\begin{thm}
Suppose that  the assumptions of Theorem \ref{fluc} hold, let $\lambda_1, \ldots, \lambda_q$ be the eigenvalues of $R$ ordered by decreasing real parts and $v_1, \ldots, v_q$ corresponding right eigenvectors of $R$.
With $p:=\max\{k \in \{1, \ldots, q\}: \Re(\lambda_k)/r
> 1/2\}$, there are two cases:
\begin{enumerate}
\item Suppose that for all $\lambda_k$ that belong to a dominant class, $\Re(\lambda_k)
  \not= r/2$. Then there exist complex--valued mean zero random variables $\Xi_1,
  \ldots, \Xi_{p}$ such that
\begin{align}
\frac{1}{\sqrt{n}}\left(X_n -\E[X_n] - \sum_{k=1}^p
  n^{\frac{\lambda_k}{r}}\Xi_kv_k \right)
  \stackrel{\mathcal{L}}{\longrightarrow} \mathcal{N}(0, A_V), \qquad n \to \infty,
\end{align}
where $\mathcal{N}$ has a non--degenerated,
centered multivariate Gaussian mixture
distribution. 
\item Suppose that there is some $\lambda_k$ that belongs to a dominant class such that $\Re(\lambda_k)
  = r/2$. Then there exist complex--valued mean zero random variables $\Xi_1,
  \ldots, \Xi_{p}$ such that
\begin{align}
\frac{1}{\sqrt{n\log n}}\left(X_n - \E[X_n] - \sum_{k=1}^p
  n^{\frac{\lambda_k}{r}}\Xi_kv_k \right)
  \stackrel{\mathcal{L}}{\longrightarrow} \mathcal{N}(0, A_V), \qquad n \to \infty,
\end{align}
where $\mathcal{N}$ has a non--degenerated,
centered multivariate Gaussian mixture
distribution. 
\end{enumerate}
\end{thm}

For a more detailed version, see Theorem \ref{fluc}.
Intuitively, the statement is the following:
For
matrices $R$ with simple largest eigenvalue $r>0$ and real parts of all
other eigenvalues bounded above by $r/2$, the correct normalisation is
the classical one. In this case, the theorem reduces to previously known results. But Theorem \ref{fluc} states that also in the more complex situation where $r$ is simple and there are other
eigenvalues with real parts greater than $r/2$, the composition vector centered by a
random vector and scaled by $\sqrt{n}$ is asymptotically normally distributed. The random variables $\Xi_1, \ldots, \Xi_p$ in Theorem \ref{fluc} arise as martingale limits and are therefore usually characterised by distributional identities \cite{J, Mailler}. Eventually, if $r$ is a multiple eigenvalue, the urn composition weakly
converges to a mixed normal distribution after centering by a random
vector and scaling by $\sqrt{n}$.
The relaxation of the term ``central limit
theorem'' therefore leads to a unified perspective for the fluctuation of urn composition vectors. From this point of view, the distinction between the different kinds of asymptotics in the urn composition arises from a decomposition of the process into a sum over components of different
sizes. These components are typically classified as {\em large} and {\em small} components. In order to treat the components on a common scale
$\sqrt{n}$, each component has to be centered in an
appropriate way. However, large components exhibit almost sure random tendencies of magnitude larger than $\sqrt{n}$, and therefore in their presence, the centering involves the deterministic expectation as well as
random terms. 

From the point of view of applications, the probably most interesting aspect of Theorem \ref{fluc} is the central limit
theorem analogue for urn schemes with periodicities, cf.~ \cite{CP2, ChHw, MN}. For the case of cyclic urns, which may serve as a prototypical example for this class, Neininger and the present author \cite{MN} use the techniques developed in \cite{Nei} to prove a central limit theorem analogue for this urn model by means of the contraction method. This strategy is of
independent interest, as it does not use underlying martingale
structures, see also \cite{KnaNe}. 

Martingales and a spectral decomposition of the process play an important role in many works on urn models \cite{athreya1968embedding, Basak, Berti, CP, CP2, Gou, J, Jan16, Mailler, MN, Pou, S}. The proof strategy of the present article starts with these common themes, but then investigates the fine fluctuations around random drifts and
random periodic growths by connecting a careful study of appropriate
residuals with classical martingale limit theory, see \cite{HH}. 



The paper is organised as follows. In section 2.1, the urn models
under consideration are defined formally and their normal form is
given. Section 2.2 contains Gouet's result on the asymptotics of the
proportions of the urn composition, as well as the central limit
theorem \ref{fluc}. We provide some examples to illustrate the central
limit theorem in section 2.3. The third section is devoted to a
careful study of the various components of the urn process and
prepares section 4, where Theorem \ref{fluc} is finally proven.

{\bf Notation.} For a complex number $z \in \mathbb{C}$, we denote by $\Re(z)$,
$\Im(z)$ and $|z|$ its real part, imaginary part and complex modulus,
respectively. For a complex vector $v \in \mathbb{C}^q$ and $i \in
\{1, \ldots, q\}$, we denote by $v^{(i)}$ its $i$-th component and by
$v^t$ and $v^{\ast}$ its transpose and conjugate transpose,
respectively. Further let $|v|$ denote its $L^1$-norm. We equip
$\mathbb{C}^q$ with the standard inner product $\langle \cdot, \cdot
\rangle$, where $\langle u, v \rangle := u^{\ast} v$. We denote by $\text{Id}_{\Cset^q}$ the $q
\times q$ identity matrix. Let $\mathbf{1}$ denote
the $q$ dimensional all ones vector. For $A \subseteq \{1, \ldots, q\}$ and $v \in \mathbb{C}$, let
$v_{A}$ be the $q$ dimensional vector defined by
$v_{A}^{(i)} = v^{(i)} \cdot \mathds{1}_{\{i \in A\}}$.  Let $\Nset:=\{1, 2, \ldots
\}$, $\Nset_0:=\{0, 1, 2, \ldots
\}$ and $\Zset_{-}:=\{0, -1 ,
-2, \ldots\}$
denote the set of non-positive integers. We use Bachmann--Landau
symbols in asymptotic statements. Finally, convergence in distribution
is denoted as $\stackrel{\mathcal{L}}{\longrightarrow}$.

\vspace{0.5 cm}
{\bf Acknowledgements.} I warmly thank Ralph Neininger and Henning Sulzbach for their help that led to a considerable improvement of the manuscript.
I am also indebted to Kevin Leckey and Andrea
Kuntschik for valuable remarks.

\section{Main result} \label{Main}

\subsection{Preliminaries}
To begin with, we specify the urn models that are the topic of the
current work and---in view of the proof---choose a suitable common
framework. For more general and more exhaustive information on urn
models, see the pivotal work by Janson \cite{J}.

We consider a generalised P\'{o}lya urn process in discrete time, which describes the joint evolution of a population of balls of $q \in \Nset$
different colours subject to drawing and replacement activities. For each $n \geq 0$ and $j \in \{1, \ldots, q\}$, let $X_{n}^{(j)}$ denote the number of balls of colour $j$ in the urn after
$n$ draws. We collect these numbers in the 
\textit{composition vector} $X_n \in \Nset_0^q$ of the urn at time $n$, i.e. 
\begin{eqnarray*}
X_n=\left(X_{n}^{(1)}, X_{n}^{(2)}, \ldots,  X_{n}^{(q)}\right)^t.
\end{eqnarray*}

The evolution of the urn process is determined by the following rules:
For each $j \in \{1, \ldots, q\}$, the number $X_{0}^{(j)}$ of balls of colour $j$ in the urn at time $0$ is fixed. In other words,
$X_0 \in \Nset_0^q$ is assumed to be any {\em deterministic} vector that satisfies some basic requirements throughout the paper. Immediately before each of the following times $n \geq 1$, a ball is
drawn uniformly at random from the urn, independently of all previous draws. If the drawn ball has colour $i$, we  add $\Delta_{i}^{(j)}$ 
balls of colour $j$ for $j=1, \ldots, q$ to the urn. Here, vector $\Delta_i \in \Zset^q$, $i=1,\ldots,q$,
coordinates the change in the urn composition if a ball of colour $i$ is
drawn and we assume $\Delta_1, \ldots, \Delta_q$ to be deterministic as well. Moreover, let $R:= (\Delta_1, \ldots,
\Delta_q)$ to be the matrix with columns
$\Delta_1, \ldots, \Delta_q$. $R$ is called the \textit{generating matrix}
of the process. Note that some authors prefer to work with the matrix transpose $R^t$, the so called \textit{replacement matrix}, which should be kept in mind when comparing the results.

The dynamics of the Markov process $(X_n)_{n \geq 0}$ are fully described by a valid $R$
and $X_0$. Our
basic assumptions on these quantities are:
\begin{enumerate}[itemsep =0pt]
\item[(A$1$)] $R$ has constant column sum $r \in \mathbb{N}$. 
\item[(A$2$)] $R_{i,j} \geq 0$ for $i \not= j$ and if $R_{i,i}<0$, then
  $|R_{i,i}|$ divides $X_{0}^{(i)}$ and $R_{i,j}$ for all $1 \leq j \leq q$.
 \item[(A$3$)] The initial composition of the urn is such that
for all colours $j$, there exists $n \in \mathbb{N}_0$ with $\Prob\left(X_{n}^{(j)} > 0\right) > 0$.
\end{enumerate}
Assumption (A$1$) guarantees a steady and deterministic linear growth, while (A$2$) is common in the literature on urn models and assures that the process does not get stuck by asking for
an impossible removal of balls. Finally, (A$3$) prevents a trivial reduction to smaller urns.

\textit{Normal form.} Matrices with non--negative
off--diagonal entries are called \textit{Metzler--Leontief} matrices, or short ML--matrices. It is possible to assume that the generating matrix $R$ is given in the following normal form, see \cite{J,S}: In order to arrive at the normal form, we first classify the indices, or colours, of $R$ in the following way (note that the partition only depends on the positions of positive entries in $R$). Write $i \to j$ (``colour $i$ leads to colour $j$'') if, starting with a single ball of colour $i$, we have
$\Prob\left(X_{n}^{(j)} > 0\right) > 0$ for some $n \in
\mathbb{N}_0$. Equivalently, $(R^n)_{j,i}>0$. We say that $i$ and $j$ communicate and write $i \leftrightarrow j$ if $i \to j$ and $j \to i$. 
The equivalence relation $\leftrightarrow$ 
partitions the set of colours $\{1, \ldots, q\}$ into equivalence classes
$\mathcal{C}_1, \ldots, \mathcal{C}_d$. 
If $d=1$, the process is called {\em irreducible}.

On the level of classes $\mathcal{C}_1, \ldots, \mathcal{C}_d$, the relation $\to$ induces a partial order: We write $\mathcal{C}_i \to \mathcal{C}_j$ (``class $\mathcal{C}_i$ leads to class $\mathcal{C}_j$''), if some (then all) colours in $\mathcal{C}_i$
lead to some (then all) colours in $\mathcal{C}_j$.
A class is called {\em dominant} if it is maximal with respect to the induced partial order, that is, if it does not lead to any other class except itself.
We distinguish three different ``types'' of classes.
Class $\mathcal{C}_i$ is of \textit{type $1$}, if it is dominant and 
there is no $j \not= i$ with $\mathcal{C}_j \to \mathcal{C}_i$.
Any dominant class which is not of type 1 is of \textit{type $2$}. All non--dominant classes are of \textit{type $3$}. Subsequently, we will always assume that the classes are ordered as follows:
Let $a$ denote the number of classes of type $1$, $c$ the number of classes of type $2$ and $b$ the number of classes of type $3$. We thus have $a, b, c \geq 0$, $a+b+c=d$ and assume that classes $\mathcal{C}_1, \ldots, \mathcal C_a$ are of type $1$, classes $\mathcal{C}_{a+1}, \ldots, \mathcal C_{a+c}$ are of type $2$, and the 
remaining classes $\mathcal{C}_{a+c+1}, \ldots, \mathcal C_{a+b+c}$ are of
the third type. $\mathcal{C}_{a+c+1}, \ldots, \mathcal C_d$ are
ordered such that $\mathcal{C}_i \to \mathcal{C}_j$ implies $i \le
j$. Note that $a+c\ge 1$. 
With respect to this ordering of the colours, the matrix $R$ takes the following lower triangular block structure:
\begin{eqnarray*}
R=\left(
  \begin{array}{cccccccccc}
    T_{1,1} &  & & & &&&&&\\
       &  \ddots  &  & & &&&&&\\
       &  &  T_{a,a} & & &&&&&\\
      & & & P_{1,1} &  & & & &&\\
       &  && \ast &\ddots  &  & & &&\\
       & && \ast & \ast & P_{b,b} & &  &  &\\
  &  & & \ast& \ast & \ast & Q_{1,1} &  & &\\
     &  && \ast& \ast & \ast & &\ddots  & & \\
       & && \ast &\ast & \ast & &  &  &  Q_{c,c} 
  \end{array}\right),
\end{eqnarray*}

Blocks
$T_{1,1}, \ldots, T_{a,a}$ correspond to type $1$ classes $\mathcal{C}_1, \ldots, \mathcal C_a$, blocks
$Q_{1,1}, \ldots, Q_{c,c}$ to type $2$ classes $\mathcal{C}_{a+1}, \ldots, \mathcal C_{a+c}$ and blocks $P_{1,1},
\ldots, P_{b,b}$ to type $3$ classes $\mathcal{C}_{a+c+1}, \ldots, \mathcal C_d$.
The middle part corresponding to type 3 classes is a lower triangular
block matrix in which 
beneath each of the blocks $P_{1,1}, \ldots, P_{b,b}$, there is at least one positive entry. Similarly, to the left
of each block $Q_{1,1},
\ldots, Q_{c,c}$, there is at least one positive entry. 

\begin{rem} (A$3$) implies that we start with at least one ball from each class of type $1$.
\end{rem}

{\em The spectrum of R.} The main point about the normal form is that all diagonal blocks are irreducible ML--matrices. Irreducible ML--matrices enjoy useful spectral properties:

\begin{thm}[\cite{Gou, Sen}]\label{MLvalues}
Let $B=(B_{i,j})_{1\leq i,j \leq q}$ be an irreducible Metzler--Leontief matrix. Then,
there exists an eigenvalue $\tau$ of $B$ such that
\begin{enumerate}
\item[(i)] $\tau$ is real, has algebraic and geometric multiplicity $1$ and the associated
  left and right eigenvectors can be chosen to have positive components;
\item[(ii)] $\tau > \Re(\lambda)$ where $\lambda \not= \tau$ is any
  eigenvalue of $B$;
\item[(iii)] $\min_j \sum_{i=1}^q B_{i,j} \leq \tau \leq \max_j \sum_{i=1}^q
  B_{i,j}$;
\item[(iv)] $\sum_{i=1}^q B_{i,j} \leq x$ for all $j$ with at least one
  strict inequality implies $\tau < x$ ($x \in \mathbb{R}$).
\end{enumerate}
\end{thm}

The spectral properties of $R$ relate to the spectral properties of its irreducible diagonal blocks in the following way. First, the eigenvalues of $R$ are given by the union of the eigenvalues of $T_{1,1}, \ldots,
T_{a,a}$, $P_{1,1}, \ldots, P_{b,b}, Q_{1,1}, \ldots, Q_{c,c}$, and we may now formulate the last assumption that is needed for Theorem \ref{fluc}, which is
\begin{enumerate}[itemsep =0pt]
\item[(A$4$)] Every submatrix $T_{1,1}, \ldots,
T_{a,a}$, $P_{1,1}, \ldots, P_{b,b}, Q_{1,1}, \ldots, Q_{c,c}$ is diagonisable over $\mathbb{C}$. If there is only one dominant class, we additionally assume its eigenvalues are pairwise distinct.
\end{enumerate}
This assumption is essential for our proof technique, but still satisfied in many applications.  

Due to assumption (A$4$), $R$ is diagonisable over $\Cset$ and thus has $q$ eigenvalues $\lambda_1,\ldots, \lambda_q \in \Cset$, if each eigenvalue is counted according to its (algebraic $=$ geometric) multiplicity.
As the columns of $T_{1,1}, \ldots,
T_{a,a}$, $Q_{1,1}, \ldots, Q_{c,c}$ sum to $r$, and the columns of
$P_{1,1}, \ldots, P_{b,b}$ sum to less than $r$, Theorem \ref{MLvalues} asserts the existence of  exactly $a+c$ ``dominant'' eigenvalues, and we impose the order 
\begin{align} \label{Order_eig}
r = \lambda_1 = \cdots = \lambda_{a+c} > \Re(\lambda_{a+c+1}) \geq
\cdots \geq \Re(\lambda_{q}). 
\end{align}
on the $q$ eigenvalues of $R$. 
Non--dominant eigenvalues with 
equal real part
are ordered by decreasing size of imaginary parts. If eigenvalue
$\lambda$ has multiplicity $m >1$, $\lambda$ is repeated $m$ times in
this list. 

We may then choose a basis $\{u_1, \ldots, u_q\}$ of $\Cset^q$ with the following properties:
\begin{itemize}
\item[(B$1$)] For all $i=1, \ldots, q$, $u^\ast_iR = \lambda_i u_i^\ast$. That is, $u_1^\ast, \ldots, u_q^\ast$ are left eigenvectors of $R$.
\item[(B$2$)] The vectors corresponding to eigenvalue $r$ take the following form: With the notation from the end of the introduction, 
\begin{align*}
u_i = \mathbf{1}_{\mathcal{C}_i} \hspace{1 cm} \text{for} \hspace{0.2 cm}  i = 1, \ldots, a. 
\end{align*}
For $s=1,\ldots, c$, there is a vector $v_s \in \mathbb{R}^q$ which is only non--zero on colour classes of type $3$ 
and
\begin{align*} 
u_{a + s} = \mathbf{1}_{\mathcal{C}_{a+b+s}} + v_s,
\end{align*}
such that $u_{a+1}, \ldots,
u_{a+c}$ are orthogonal. 
Thus, $u_1, \ldots, u_{a+c}$ are orthogonal. 
\item[(B$3$)] Generally, basis vectors come from eigenvectors of blocks that are extended to eigenvectors of $R$ in the following way:  
\begin{itemize}
\item[(i)] If $\lambda$ is an eigenvalue of multiplicity $m$ of $T_{i,i}$, $1 \leq i \leq a$, the $m$ corresponding left basis vectors are zero on every index outside $T_{i,i}$. 
\item[(ii)] If $\lambda$ is an eigenvalue with multiplicity $m$ of 
$P_{i,i}$, $1 \leq i \leq b$, the $m$
corresponding basis vectors are zero on
colours in type $1$ and type $2$ classes.
\item[(iii)]  Similarly, if $\lambda$ is an eigenvalue with multiplicity $m$ of 
$Q_{i,i}$, $1 \leq i \leq c$, the corresponding $m$
 basis vectors are 
 zero on
all colours in type $1$ blocks, in type $2$ blocks $Q_{j,j}$ for $j
\in \{1, \ldots, c\}\setminus \{i\}$ 
\end{itemize}
\item[(B$4$)] Eigenvectors corresponding to real eigenvalues are chosen to have real components. Moreover, if $\lambda_k \in \Cset \setminus \Rset$ is an
eigenvalue with corresponding eigenvector $u_k^\ast$, then for $\lambda_\ell = \bar{\lambda}_k$ we choose $u_\ell^\ast=\bar{u}_k^\ast$. 
\end{itemize}
Let  $\{v_1, \ldots, v_q\}$ be the basis dual to $\{u_1, \ldots, u_q\}$, i.e., for all $i=1, \ldots, q$,
\begin{align*}
u_i^\ast v_i = \delta_{ij}.
\end{align*}
The basis $\{v_1, \ldots, v_q\}$ has the following analoguos properties: 
\begin{itemize}
\item For all $i=1, \ldots, q$, $Rv_i = \lambda_i v_i$. That is, $v_1, \ldots, v_q$ are right eigenvectors of $R$.
\item Eigenvectors corresponding to real eigenvalues are chosen to have real components. Moreover, if $\lambda_k \in \Cset \setminus \Rset$ is an
eigenvalue with corresponding eigenvector $v_k$, then for $\lambda_\ell = \bar{\lambda}_k$, $v_\ell=\bar{v}_k$. 
\item  If $\lambda$ is an eigenvalue of multiplicity $m$ of $T_{i,i}$, $1 \leq i \leq a$, the $m$ corresponding basis vectors are zero on every index outside $T_{i,i}$ 
. If $\lambda$ is an eigenvalue with multiplicity $m$ of 
$P_{i,i}$, $1 \leq i \leq b$, the $m$
corresponding basis vectors are zero on
colours in type $1$. 
Similarly, if $\lambda$ is an eigenvalue with multiplicity $m$ of 
$Q_{i,i}$, $1 \leq i \leq c$, the corresponding $m$
 basis vectors 
 are zero on
all colours outside $Q_{i,i}$. 
\end{itemize}

With this particular choice of bases $\{u_1, \ldots, u_q\}$ and $\{v_1, \ldots, v_q\}$, we can decompose $R$ as 
\begin{align*}
R = \lambda_1 v_1 u_1^\ast + \cdots + \lambda_q  v_q u_q^\ast.
\end{align*}
Moreover, let $\pi_{k}: \Cset^q \to \Cset$ be the linear map defined by
\begin{align*}
\pi_k(v) := u_k^{\ast}  v.
\end{align*}

\subsection{A central limit theorem}

Let us begin with a well--known result on strong convergence of proportions. 
Assumption (A$1$) ensures that at any time, the total number of balls in the urn is deterministic and given by $|X_n|=|X_0|+rn$. The asymptotic share of the $q$ colours in these $|X_0|+rn$ balls is the content of the following theorem.

\begin{thm}[{\cite[Theorem 3.1]{Gou}} and {\cite[Theorem 3.5]{Pou}}] \label{Assi}
Suppose that (A$1$) and (A$2$) hold and
that $c=1$. Then, as $n \to \infty$,
\begin{align}\label{Dirichlet_limit}
\frac{X_n}{rn+|X_0|} \to \sum_{i=1}^a D^{(i)} v_i + D^{(a+1)}v_{a+1} \qquad \text{a.s.},
\end{align}
where $\left(D^{(1)}, \ldots, D^{(a)},
D^{(a+1)}\right)^t$ is Dirichlet distributed with parameter
\begin{align*}
\theta = \left(\frac{|(X_{0})_{\mathcal{C}_1}|}{r}, \ldots, \frac{|(X_{0})_{\mathcal{C}_a}|}{r}, \frac{|X_0|}{r}
- \sum_{j=1}^a \frac{|(X_{0})_{\mathcal{C}_j}|}{r} \right).
\end{align*}
\end{thm}

In order to develop an intuitive understanding of Theorem \ref{Assi}, it is fruitful to  compare it to two special cases, namely an irreducible urn and a P\'olya urn. In the first case, there is only one irreducible class and the long--time proportion of each colour in the urn composition is deterministically given by the corresponding component of the properly normalised dominant eigenvector, irrespective of the initial configuration. In the second case, however, each colour forms its own irreducible type $1$ class. The asymptotic proportions are given by a Dirichlet distributed random vector, which is highly sensitive to the initial configuration. 
Theorem \ref{Assi} locates the asymptotics of a more general urn model in between the two special cases: The dominant classes $\mathcal{C}_1, \ldots, \mathcal{C}_a, \mathcal{C}_{a+1}$ act as ``supercolours'' in a P\'{o}lya urn; there is no exchange of balls between them. In line with our preceding observations,
the asymptotic proportions among these supercolours are
Dirichlet distributed. On the other hand, the way in which each Dirichlet component further splits among the colours of a particular dominant and irreducible class are
deterministic and given by the components of the right
eigenvector corresponding to the class. Finally, the asymptotic proportions of balls of non--dominant classes are zero almost surely.

However, Theorem \ref{Assi} only covers the case $c=1$. In the more general case $c \in \Nset$, the above result makes it plausible that 
\begin{align*}
\frac{X_n}{rn+|X_0|} \to \sum_{i=1}^a D^{(i)}v_i + D^{(a+1)}(\Gamma_{a+1}v_{a+1} +
  \cdots + \Gamma_{a+c}v_{a+c}) 
\end{align*}
almost surely, where $\left(D^{(1)}, \ldots, D^{(a)},
D^{(a+1)}\right)^t$ is Dirichlet distributed with parameter $\theta$ as in
Theorem \ref{Assi}. $\Gamma_{a+1}, \ldots, \Gamma_{a+c}$ are random
variables that sum to $1$ almost surely and are independent of the Dirichlet
random vector. Intuitively,  the random variables $\Gamma_{a+1},
\ldots, \Gamma_{a+c}$ are the
asymptotic proportions of the non--isolated dominant classes inside supercolour $\mathcal{C}_{a+1} \cup
\ldots \cup \mathcal{C}_d$. A proof of this result can easily be obtained along the
lines of the proofs given in the next section, and we omit the details
here and move on to our main theorem.

In the following, we work with the centered sequence of urn compositions: For $n \geq 0$, set 
\begin{align*}
Y_n := X_n - \E[X_n].
\end{align*}
The random vector $V$ denotes the
almost sure limit of the proportions $\frac{X_n}{rn+|X_0|}$,
\begin{align*}
V:= \sum_{i=1}^a D^{(i)}v_i + D^{(a+1)}(\Gamma_{a+1}v_{a+1} +
  \cdots + \Gamma_{a+c}v_{a+c}).
\end{align*}
Recall that
it is zero in all type $3$ components. Finally, we define the matrix
\begin{align*}
M :=  \left(\Re(v_1), -\Im(v_1), \Re(v_2), -\Im(v_2), \ldots,
  \Re(v_q), -\Im(v_q) \right) \in \Rset^{q \times 2q}.
\end{align*}


\begin{thm}\label{fluc}
Suppose that  assumptions (A$1$) - (A$4$) hold, the eigenvalues of $R$ are ordered as in (\ref{Order_eig}), $\{u_1, \ldots, u_q\}$ is a basis of $\mathbb{C}^{q}$ that satisfies (B$1$) - (B$4$) and $\{v_1, \ldots, v_q\}$ is the dual basis. Let
$p:=\max\{k \in \{1, \ldots, q\}: \Re(\lambda_k)/r
> 1/2\}$. There are two cases:
\begin{enumerate}
\item Suppose that for all $\lambda_k$ that belong to a dominant class, $\Re(\lambda_k)
  \not= r/2$. Then there exist complex--valued mean zero random variables $\Xi_1,
  \ldots, \Xi_{p}$ such that
\begin{align}
\frac{1}{\sqrt{n}}\left(Y_n - \sum_{k=1}^p
  n^{\frac{\lambda_k}{r}}\Xi_kv_k \right)
  \stackrel{\mathcal{L}}{\longrightarrow} \mathcal{N}(0, A_V)
\end{align}
as $n \to \infty$, where $\mathcal{N}$ denotes a non--degenerated,
centered multivariate Gaussian mixture
distribution with latent distribution $\mathcal{L}(V)$ and
covariance matrix 
\begin{align*}
A_V := & M  \Sigma_V M^t,
\end{align*}
 where $\Sigma_V$ is
defined in (\ref{Sigma1}) to (\ref{Sigma4}) below Theorem
\ref{CW}. Furthermore, $(A_V)_{i,i}>0$
almost surely for dominant colours $i$, whereas $(A_V)_{i,i}=0$ almost
surely for non dominant colours $i$.
\item Suppose that there is some $\lambda_k$ that belongs to a dominant class such that $\Re(\lambda_k)
  = r/2$. Then there exist complex--valued mean zero random variables $\Xi_1,
  \ldots, \Xi_{p}$ such that
\begin{align}
\frac{1}{\sqrt{n\log n}}\left(Y_n - \sum_{k=1}^p
  n^{\frac{\lambda_k}{r}}\Xi_kv_k \right)
  \stackrel{\mathcal{L}}{\longrightarrow} \mathcal{N}(0, A_V)
\end{align}
as $n \to \infty$, where $\mathcal{N}$ denotes a non--degenerated,
centered multivariate Gaussian mixture
distribution with latent distribution $\mathcal{L}(V)$ and
covariance matrix 
\begin{align*}
A_V := & M  \Sigma_V M^t,
\end{align*}
 where $\Sigma_V$ is
defined in (\ref{Sigma5}) below Theorem \ref{CW}. $(A_V)_{i,i}>0$
almost surely for dominant colours $i$ that belong to the irreducible
classes of eigenvalues with real part $r/2$, whereas $(A_V)_{i,i}=0$ almost
surely for all other colours.
\end{enumerate}
\end{thm}

\begin{rem}\label{irr_case} Central limit theorems in the case where the dominant eigenvalue $r$ is simple and $\Re(\lambda_k) \le r/2$ for all other eigenvalues $\lambda_k$ are well--known, see e.g. \cite{J} and \cite{S}. 
\end{rem}

Results on urn models are usually stated separately for three main classes: First, models where the rescaled composition vector converges to a deterministic limit and is asymptotically normally
distributed (e.g., $m-$ary search tree for $m \leq 26$), second, where it converges almost surely to a random limit
(e.g., P\'olya urn) and third, where it exhibits almost sure oscillating
behaviour (e.g., $m-$ary search tree for $m \geq 27$ or cyclic urn for
$m \geq 7$ colours). Theorem \ref{fluc} provides a common framework for these cases (and others). Its main structural content is the following: 

First, the centering of the urn composition may be random in order to obtain a central limit theorem. Whether it is random or not, depends on whether $\Re(\lambda_2)>r/2$, where the possibility that $\lambda_2=r$ is included. If $\Re(\lambda_2)\leq r/2$, the centering is deterministic, corresponding to Remark \ref{irr_case}. Otherwise, it is random. The random variables that need to be substracted arise from non--negligible drifts in the urn composition that surpass the normal $\sqrt{n}$ or $\sqrt{n \log n}$ fluctuation.

Second, the magnitude of the fluctuation is determined by the existence of eigenvalues $\lambda_k$ belonging to a dominant class such that $\Re(\lambda_k)
  = r/2$. If none such eigenvalue exists, the fluctuation is of order $\sqrt{n}$. Otherwise, it is of order $\sqrt{n \log n}$. In the case of a simple dominant eigenvalue $r$, this is a well--known phenomenon, see \cite{J}.
  
 Third, the limiting distribution that arises is not necessarily Gaussian, but rather mixed Gaussian with random covariance structure. The covariance matrix of the limiting Gaussian distribution depends on the components of $V$, the asymptotic proportions of supercolours. Again, this behaviour basically transfers from a central limit theorem for the P\'olya urn by regarding irreducible dominant classes as supercolours.
 
Finally, the special form $A_V =  M  \Sigma_V M^t$ arises from the fact that we work with projections of the urn process that are transformed back to the original process by the linear transformation $M$. The matrix $\Sigma_V$ is given in Section \ref{Sec_proof} and also depends on the choice of basis, but has interesting structural properties.

\subsection{Applications}

To illustrate the statement of Theorem \ref{fluc}, we give four examples that cover both urns with $\Re(\lambda_2) >r/2$ and with $\Re(\lambda_2)\leq r/2$. These examples are particularly interesting, as they are subject to a phase change when parametrized as below. Subsections \ref{Polyaurn} and \ref{Friedmanurn} are covered in the existing literature, while the results in subsections \ref{m-ary} and \ref{B-tree} are new. In particular, subsection \ref{m-ary} refines the asymptotics of the size of $m$-ary search trees, an important subject in the probabilistic analysis of algorithms started by Knuth \cite{Knu3}. 

\subsubsection{P\'olya urn} \label{Polyaurn}
Consider the P\'olya urn with $q \geq 2$ colours and matrix
$R=r \cdot\text{Id}_{\mathbb{R}^q}$. The eigenvalues of $R$ are given by $\lambda_1 = \cdots =\lambda_q=r$, and in particular, $r$ is a multiple eigenvalue.  We choose $\{u_1, \ldots, u_q\}$ to be the canonical basis of $\mathbb{C}^q$. This choice obviously satisfies (B$1$) to (B$4$) and yields $v_i=u_i$ for $i=1, \ldots, q$. With initial configuration $X_0 \in \mathbb{N}^q$, $X_0^{(i)} >0$ for $i=1, \ldots, q$, Theorem \ref{Assi} implies that
\begin{flalign*}
&& \frac{X_n}{rn+|X_0|} \stackrel{\text{a.s.}}{\longrightarrow} V, && n \to \infty,
\end{flalign*}
where $V=\left(V^{(1)}, \ldots, V^{(q)} \right)^t$ is a random Dirichlet vector with parameter $\theta$ as in Theorem \ref{Assi}. Moreover, the evaluation of $\Sigma_V$ as in Theorem \ref{fluc} yields the following central limit theorem:
\begin{align*}
\frac{1}{\sqrt{r^2n}}\left(X_n - nV\right) \stackrel{\mathcal{L}}{\longrightarrow} \mathcal{N}\left(0, \begin{pmatrix}
V_1(1-V_1) & -V_1V_2& \cdots & -V_1V_q \\
-V_1 V_2 & V_2(1-V_2) &\cdots & -V_2V_q \\
\vdots &        \vdots     &   \ddots      & \vdots\\
-V_1V_q & -V_2V_q & \cdots & V_q(1-V_q)
\end{pmatrix} \right).
\end{align*} 
Note that because the almost sure limit of $X_n/(rn+|X_0|)$ is random, a mixed normal
distribution arises, as noted below Theorem \ref{fluc}. 
It is instructive to bear this covariance structure in mind, as it is the basis for generating matrices with a more complex eingevalue structure. For $q=2, r=1$ and initial configuration $X_0=(1,1)^t$, we recover a result of Hall and Heyde \cite{HH} p. $80$.


\subsubsection{Friedman's urn} \label{Friedmanurn}
As a two--colour extension of the previous example, consider Friedman's urn
with generating matrix 
\begin{align*}
R= \left(\begin{array}{cc}
\alpha & \beta \\
\beta & \alpha
\end{array}\right),
\end{align*}
where $\alpha, \beta \in \Zset$, $\alpha \geq -1$, $\beta \geq 0$, and
$\alpha + \beta = r >0$. This symmetric matrix has real eigenvalues $\lambda_1:=
\alpha+\beta$ and $\lambda_2:=\alpha-\beta$. They are distinct unless $\beta = 0$, which is the original P\'olya
urn from the previous example. For $\beta > 0$, the urn is irreducible and we choose eigenvectors 
\begin{align*}
u_1:=\left(\begin{array}{c}
1\\
1
\end{array} \right), \hspace{0.3 cm} u_2:=
\left(\begin{array}{c}
1\\
-1
\end{array} \right)
\end{align*}
which satsify (B$1$) to (B$4$). This choice yields
\begin{align*}
v_1:=\frac{1}{2}\left(\begin{array}{c}
1\\
1
\end{array} \right), \hspace{0.3 cm} v_2:=
\frac{1}{2}\left(\begin{array}{c}
1\\
-1
\end{array} \right).
\end{align*}
Within the case $\beta>0$, three different limiting scenarios arise, corresponding to three different positions of the second largest eigenvalue with respect to $\lambda_1/2$: If $\alpha < 3\beta$, then $\lambda_2 < \lambda_1 /2$, and 
\begin{align*}
A_V &= M \Sigma_V M^t
=\begin{pmatrix} \frac{1}{2} & 0 & \frac{1}{2} & 0 \\
\frac{1}{2} & 0 & -\frac{1}{2} & 0 
\end{pmatrix}
                                 \frac{(\alpha+\beta)(\alpha-\beta)^2}{3\beta-\alpha} \begin{pmatrix}
                                   0 & 0 & 0 & 0\\ 0 & 0 & 0 & 0\\
                                 0 & 0 & 1 & 0\\
                               0 & 0 & 0 & 0\\
\end{pmatrix}\begin{pmatrix}
\frac{1}{2} & \frac{1}{2}\\
0 & 0 \\
\frac{1}{2} & -\frac{1}{2} \\
0 & 0
\end{pmatrix}\\
 &= \frac{(\alpha + \beta)(\alpha-\beta)^2}{4(3\beta-\alpha)}\begin{pmatrix} 1 & -1 \\ -1 &
    1 \end{pmatrix},
\end{align*}
so
\begin{flalign*}
&& \frac{Y_n}{\sqrt{n}} \stackrel{\mathcal{L}}{\longrightarrow}
  \mathcal{N}\left(0, \frac{(\alpha + \beta)(\alpha-\beta)^2}{4(3\beta-\alpha)}\begin{pmatrix} 1 & -1 \\ -1 &
    1 \end{pmatrix}\right), && n \to \infty.
\end{flalign*}
If $\alpha = 3\beta$, then
$\lambda_2 = \lambda_1 /2$, a very similar calculation leads to
\begin{flalign*}
&& \frac{Y_n}{\sqrt{n\log n}} \stackrel{\mathcal{L}}{\longrightarrow}
  \mathcal{N}\left(0, \frac{(\alpha-\beta)^2}{4}\begin{pmatrix} 1 & -1 \\ -1 &
    1 \end{pmatrix}\right),&& n \to \infty.
\end{flalign*}
The preceding two central limit theorems were obtained in \cite{Bernstein, Free} and appear as example
$3.27$ in \cite{J}.

Finally, if $\alpha > 3\beta$, 
$\lambda_2 >\lambda_1 /2$, and a random centering yields
\begin{flalign*}
&& \frac{Y_n - n^{\frac{\alpha-\beta}{\alpha+\beta}}\Xi_2 v_2}{\sqrt{n}}
  \stackrel{\mathcal{L}}{\longrightarrow} \mathcal{N}\left(0,
  \frac{(\alpha + \beta)(\alpha-\beta)^2}{4(\alpha- 3\beta)}\begin{pmatrix} 1 & -1 \\ -1 &
    1 \end{pmatrix}\right), && n \to \infty,
\end{flalign*}
where $\Xi_2$ is the almost sure limit of the martingale
\begin{align*}
\left(\frac{\Gamma\left(n+\frac{|X_0|}{\alpha + \beta} \right)}{\Gamma\left(n+\frac{|X_0|+\alpha-\beta}{\alpha + \beta} \right)} \left(Y_n^{(1)} - Y_n^{(2)}\right) \right)_{n \geq 0}.
\end{align*}
In all three cases, due to the fact that for $\beta > 0$, $\lambda_1=r$ is simple, no mixing over normal distributions arises.

\subsubsection{\texorpdfstring{$m$}{}-ary search tree} \label{m-ary}
Used for searching and sorting of linearly ordered data, $m$-ary search trees are fundamental data structures in computer science, cf.~ \cite{Knu1, Knu3, M2}. Each node of an $m$-ary search tree may contain $0$ to $m-1$ keys, and we refer to a node containing $i \in \{0, \ldots, m-2\}$ keys as a node of type $i$. It is observed in \cite{M1} that the joint evolution of a linear transformation of the various node types of an $m$-ary search tree generated by the uniform permutation model can be regarded as an urn model. More precisely, if $X_n^{(i)}, i=0 , \ldots, m-2$, denotes $(i+1)$ times the number of nodes of type $i$ after the insertion of $n$ keys, the dynamics of the vector $X_n$ are given by an urn model with $X_0=(1,0, \ldots, 0)^t$ and irreducible generating matrix
\begin{eqnarray*}
R_m=\left(
  \begin{array}{ccccc}
    -1 & 0 & &  & m\\
     2  & -2  &  & & \\
       &3 & -3 & & \\
       &  &  & \ddots  & \\
      & & & m-1& -(m-1)
  \end{array}\right).
\end{eqnarray*}
The general results on irreducible urn models imply that  
\begin{flalign*}
&& \frac{X_n}{n+1} \stackrel{a.s.}{\longrightarrow} \frac{1}{H_m-1}
  \left(\frac{1}{ 2}, \frac{1}{3}, \ldots,
  \frac{1}{ m} \right)^t, && n \to \infty,
\end{flalign*}
where $H_m$ denotes the $m-$th Harmonic number, cf.~ \cite[equation $(17)$]{CP}. The $m-1$ simple eigenvalues of $R_m$ are given by the solutions of the
equation
\begin{align*}
m! = \prod_{k=1}^{m-1} (z+k).
\end{align*}
As the dominant eigenvalue $1$ is simple for all $m$, we
expect weak convergence with a non--mixed Gaussian limit when searching for central limit theorems. Indeed, if $m \leq 26$, there are no eigenvalues with real
part greater than $1/2$. In this case, $(X_n-\E[X_n])/\sqrt{n}$ converges to a normal law in distribution, and Theorem \ref{fluc} therefore confirms
the well--known result that the limiting distribution of the normalized space requirement of the $m$-ary search tree is asymptotically normally distributed, cf.~ \cite{MahPit, LewMah}. The convergence also is an immediate consequence of the aforementioned results on irreducible urn models. 
For $m > 26$, there is at least one eigenvalue with real part greater
than $1/2$ and it is known that for all such $m$, there is no
eigenvalue whose real part is equal to $1/2$. Chern and Hwang
\cite{ChHw} prove that when $m \geq 27$, the space requirement
centered by its mean and scaled by its standard deviation does not
have a limiting distribution. However, \cite{CP} show that in this case, the vector of the node types can be almost surely approximated by an oscillating sequence with random amplitude and phase shift, and for the corresponding approximation of the total number of nodes, \cite{Fill} identify the distribution of the arising complex--valued random variable by a stochastic fixed point equation. In line with these results, Theorem \ref{fluc} yields a normal fluctuation about the known strong approximations,
\begin{align*}
\frac{1}{\sqrt{n}}\left(X_n - \E[X_n] - \sum_{k = 1}^p
  n^{\lambda_k}\Xi_kv_k \right)
  \stackrel{\mathcal{L}}{\longrightarrow} \mathcal{N}(0, A_V).
\end{align*}
Nevertheless, in order to derive a central limit theorem, more and more oscillating terms need to be substracted as $m$ grows. See \cite{MN} for a similar result.

\subsubsection{B-tree} \label{B-tree}
Another example of a search tree is the so--called B-tree, introduced by Bayer and McCreight \cite{Bay1, Bay2}. Yao \cite{Yao} observed that the fringe nodes of a B-tree generated by the random permutation model can be regarded as a P\'olya urn, an aspect which is further developed in \cite{CP2}. To formulate the urn process, it is necessary to specify an algorithm by which the B-tree is generated. Chauvin, Gardy, Pouyanne and Ton-That \cite{CP2} consider two algorithms, called ``prudent'' and ``optimistic''. For the purpose of illustrating Theorem \ref{fluc}, we only state the results in the ``optimistic'' case, even though everything holds for the prudent case as well. Similar to the $m$-ary search tree, the B-tree is defined by a parameter $m \in \mathbb{N}_{\geq 2}$ which determines the capacity of the nodes. The nodes of the tree, whose only descendants are leaves, are called fringe nodes and have different types, depending on how many keys they contain. Let $(X_n)_n$ be the gap process of the fringe, that is, for $i=1, \ldots, m$, $X_n^{(i)}$ is $(m+i-1)$ times the number of fringe nodes holding $m+i-2$ keys after the insertion of $n$ keys.  
The process $(X_n)_n$ can be regarded as a P\'{o}lya urn with $X_0=(m,0,\ldots,0)^t$ and irreducible
generating matrix
\begin{eqnarray*}
R_m=\left(
  \begin{array}{ccccc}
    -m &  & &  & 2m\\
     m+1  & -(m+1)  &  & & \\
       &m+2 & -(m+3) & & \\
       &  &  & \ddots  & \\
      & & & 2m-1& -(2m-1)
  \end{array}\right).
\end{eqnarray*}
The $m$ simple eigenvalues of $R_m$ are given by the solutions of the
equation
\begin{align*}
\frac{2m!}{m!} = \prod_{k=m}^{2m-1} (z+k).
\end{align*}
Left and right eigenvectors of $R_m$ are explicitly calculated in \cite[equation $(9)$]{CP2}. If $m \leq 59$, there are no eigenvalues with real
part greater than $1/2$. In this case, the results of \cite{J,S} yield that
\begin{flalign*}
&& \frac{X_n - \E[X_n]}{\sqrt{n}} \stackrel{\mathcal{L}}{\longrightarrow} \mathcal{N}(0,A_V), && n \to \infty.
\end{flalign*}
For $m \ge 60$ however, there are eigenvalues with real
part greater than $1/2$, and again, $(X_n)_n$ can be almost surely approximated by an oscillating sequence with random amplitude and phase shift \cite{CP2}. Theorem \ref{fluc} refines this result by stating that
\begin{flalign*}
&& \frac{1}{\sqrt{n}}\left(X_n - \E[X_n] - \sum_{k = 1}^p
  n^{\lambda_k}\Xi_kv_k \right)
  \stackrel{\mathcal{L}}{\longrightarrow} \mathcal{N}(0,A_V), && n \to \infty.
\end{flalign*}
Of course, the evaluation of $A_V$ uses the eigenvalues (and eigenvectors) of $R_m$, which take no simple form in this example. However, as the dominant eigenvalue $1$ is simple, there is no mixed normal distribution. See \cite{CP2} for more information on the definition of the optimistic algorithm and properties of $\Xi_2$ ($\Xi_1=0$ here). 

\section{Proof of Theorem \ref{fluc}} \label{Proof}
\subsection{Projections and martingales}
Key to the proof of Theorem
\ref{fluc} is an understanding of the asymptotics of the components of $RY_n$ in the decomposition
\begin{align*}
R Y_n = \lambda_1 v_1 \pi_1(Y_n) + \ldots + \lambda_q v_q \pi_q(Y_n).
\end{align*}
More precisely, we will study the scalar projection coefficients
$\pi_k(X_n)$ via martingale techniques, which have played an important role
in the analysis of urn models for a long time. To do so, we begin with the non--centered sequence $(X_n)_{n \geq 0}$ and denote the canonical filtration of the urn process by
\begin{align*}
\mathcal{F}_n:=\sigma(X_0, \ldots, X_n).
\end{align*}
It is immediate that
\begin{align} \label{bed}
\E[X_{n+1} | \mathcal{F}_n] = \left( \text{Id}_{\Cset^q} +
  \frac{R}{rn+|X_0|}\right) X_n, 
\end{align}
yielding a vector--valued
martingale
\begin{align*}
\left( \prod_{j=N}^{n-1} \left( \text{Id}_{\Cset^q} +
  \frac{R}{rj+|X_0|}\right)^{-1} X_n\right)_{n \ge N}.
\end{align*}
Here, $N \in \Nset$ is chosen sufficiently large such that the occurring
matrix inverses exist. 
It is this particular form of $\E[X_{n+1}|\mathcal{F}_n]$ that leads to complex--valued
martingales via projections on the eigenspaces of $R$. These projection martingales frequently play an important part in the
analysis of urn models, see \cite{Basak, Free, Mailler,
  S,Pou}, just to name a few. Therefore, in this section, to the author's knowledge, only Corollary \ref{variance}
and Lemma \ref{Speed} (in the general setting of the present
article) are original, while the other
results are known and only included to keep the article as self--contained as
possible. The following lemma introduces the projection martingales and
states an asymptotic expansion of their means.

\begin{lem}[Projection martingales] \label{erw} We distinguish two cases:
\begin{itemize}
\item[(i)] If $k \in \{1,
\ldots, q\}$ is such that $\lambda_k$ satisfies
$\lambda_k + |X_0| \notin r\Zset_{-}$ and $n \geq 0$, set 
\begin{align*}
\gamma_n^{(k)} := \prod_{j=0}^{n-1} \left( 1+
  \frac{\lambda_k}{rj+|X_0|}\right) \qquad \text{and} \qquad M^{(k)}_n := (\gamma_n^{(k)})^{-1} \cdot \pi_k(Y_n).
\end{align*}
Then $(M^{(k)}_n)_{n \geq 0}$ is a complex--valued martingale with
mean zero and
\begin{align*}
\E[\pi_k(X_n)] &= \gamma_n^{(k)} \pi_k(X_0) =
                 \frac{\Gamma\left(\frac{|X_0|}{r}\right)
                 \pi_k(X_0)}{\Gamma\left(\frac{|X_0|+\lambda_k}{r}
                 \right)}  \cdot n^{\frac{\lambda_k}{r}}+
O\left(n^{\Re\left(\frac{\lambda_k}{r}\right)-1}\right), \qquad n \to \infty.
\end{align*}
\item[(ii)] If $k \in \{1,
\ldots, q\}$ is such that $\lambda_k$ satisfies $\lambda_k + |X_0| \in
r\Zset_{-}$ and $n \ge -\frac{\lambda_k+|X_0|}{r}+1$, set
\begin{align*}
\gamma_n^{(k)} := \prod_{j=-\frac{\lambda_k+|X_0|}{r}+1}^{n-1} \left( 1+
  \frac{\lambda_k}{rj+|X_0|}\right) \qquad \text{and} \qquad M^{(k)}_n := (\gamma_n^{(k)})^{-1} \cdot \pi_k(Y_n).
\end{align*}
Then $(M^{(k)}_n)_{n \ge -\frac{\lambda_k+|X_0|}{r}+1}$ is a
complex--valued martingale with mean zero and for all $n \ge -\frac{\lambda_k+|X_0|}{r}+1$,
\begin{align*}
\E[\pi_k(X_n)] &= 0.
\end{align*}
\end{itemize}
\end{lem}

\begin{proof}
Let $k \in \{1, \ldots, q\}$ and $n \geq 0$. As a direct
consequence of (B$1$) and (\ref{bed}) for 
all $n \ge 0$,
\begin{align*}
\E[\pi_k(X_{n+1})|\mathcal{F}_n] 
= \left(1+ \frac{\lambda_k}{rn+|X_0|}\right) \pi_k(X_n),
\end{align*}
which implies that $\pi_k(X_n)$ can be normalised to a martingale
as long as $\gamma_n^{(k)}\not=0$. In the case $\lambda_k + |X_0| \in
r\Zset_{-}$ this can be ensured by leaving out the first few steps. Also,
\begin{align*}
\E[\pi_k(X_n)] = \prod_{j=0}^{n-1} \left( 1+
  \frac{\lambda_k}{rj+|X_0|}\right) \pi_k(X_0), 
\end{align*}
which is zero in
the second case. In the first case, by
Stirling's formula,
\begin{align*}
\gamma_n^{(k)} =
 \frac{\Gamma\left(\frac{|X_0|}{r} \right)}{\Gamma\left(\frac{|X_0|+\lambda_k}{r} \right)} \cdot \frac{\Gamma\left(n +\frac{|X_0|}{r} +
  \frac{\lambda_k}{r} \right)}{\Gamma(n+\frac{|X_0|}{r})} =  \frac{\Gamma\left(\frac{|X_0|}{r} \right)}{\Gamma\left(\frac{|X_0|+\lambda_k}{r} \right)} \cdot n^{\frac{\lambda_k}{r}}+ O\left(n^{\Re\left(\frac{\lambda_k}{r}\right)-1}\right)
\end{align*} 
 as $n \to \infty$. This implies the claim.
\end{proof}

\vspace{0.05 cm}
The martingales of the preceding proposition can be divided into two classes: convergent and non
convergent. The corresponding eigenvalues are often referred to as
``big'' and ``small'', respectively. The remainder of this section will be devoted to
properties of the convergent martingales and their limits. 

\begin{lem}[Martingale limits]\label{conv}
For each $k \in \{1, \ldots, q\}$ such that $\Re(\lambda_k)>r/2$, there
exists a complex--valued mean zero random variable $\Xi_k$ such that 
\begin{align*}
M_n^{(k)} \to \frac{\Gamma(|X_0|/r + \lambda_k/r)}{\Gamma(|X_0|/r)}
  \quad \Xi_k
\end{align*}
almost surely and in $L^2$ as $n \to \infty$.
\end{lem}

\begin{rem} The limiting random variables $\Xi_k$ yield the centering random variables for Theorem
\ref{fluc}.
\end{rem}

\vspace{0.2 cm}
\begin{proof}
We apply the $L^2$--martingale convergence theorem and show boundedness
of second moments. 
\begin{align*}
\E\left[|\pi_k(X_{n+1})|^2| \mathcal{F}_n \right] &=
\left(1+\frac{2\Re(\lambda_k)}{rn+|X_0|}\right)|\pi_k(X_n)|^2 +
                                                    \sum_{j=1}^q \frac{X_{n}^{(j)}}{rn+|X_0|}|\pi_k(\Delta_j)|^2.
\end{align*}
Set $C_k := \sum_{j=1}^q|\pi_k(\Delta_j)|^2$. With this,
\begin{align*}
\E\left[|\pi_k(X_{n+1})|^2| \mathcal{F}_n \right]
  &\leq\left(1+\frac{2\Re(\lambda_k)}{rn+|X_0|}\right)|\pi_k(X_n)|^2 + C_k
\end{align*}
and thus
\begin{align*}
& \E[|\pi_k(X_n)|^2] \\
&\leq \prod_{j=0}^{n-1}
  \left(1+\frac{2\Re(\lambda_k)}{rj+|X_0|}\right) \E[|\pi_k(X_0)|^2] + C_k
                     \prod_{j=1}^{n-1}
  \left(1+\frac{2\Re(\lambda_k)}{rj+|X_0|}\right) \sum_{m=0}^{n-1}
                     \prod_{j=1}^m
                     \left(1+\frac{2\Re(\lambda_k)}{rj+|X_0|}\right)^{-1}\\
&= \prod_{j=0}^{n-1} \left(1+\frac{2\Re(\lambda_k)}{rj+|X_0|}\right) \left(\E[|\pi_k(X_0)|^2] + C_k
  \left(1+\frac{2\Re(\lambda_k)}{|X_0|}\right)^{-1} \sum_{m=0}^{n-1}
                     \prod_{j=1}^m
                     \left(1+\frac{2\Re(\lambda_k)}{rj+|X_0|}\right)^{-1}\right)\\
&= \bo\left( n^{2\Re(\lambda_k)/r}\right)
\end{align*}
as $n \to \infty$, because $\Re(\lambda_k) > r/2$. 
Thus
\begin{flalign*}
&& \E[|M_n^{(k)}|^2] \leq  |\gamma_n^{(k)}|^{-2} \E\left[|\pi_k(X_n)|^2\right]
 = \bo(1), && n \to \infty.
\end{flalign*}
By
the $L^2$--martingale convergence theorem, $M_n^{(k)}$ converges almost
surely and in $L^2$ to a complex random variable which we write as $\frac{\Gamma(|X_0|/r + \lambda_k/r)}{\Gamma(|X_0|/r)}
  \quad \Xi_k$. 
\end{proof}

\begin{rem}[Asymptotic proportions via martingale limits] Recall the definition of the random proportions $D^{(1)}, \ldots,
D^{(a+1)}\Gamma_{a+1}, \ldots, D^{(a+1)}\Gamma_{a+c}$ in
Theorem \ref{Assi} and below. For $k
\in \{1, \ldots, a\}$, taking the scalar product with $u_k$ in (\ref{Dirichlet_limit}) immediately yields
\begin{align}\label{rel}
\frac{\Xi_k}{r} = D^{(k)} - \frac{\pi_k(X_0)}{|X_0|}.
\end{align}
This is true even for $c>1$, which can be seen by regarding all balls of types $2$ and $3$ as balls of just one equal colour. Therefore, the limiting proportion of balls of types $2$ and $3$ is given by the Dirichlet component $D^{(a+1)}$.
More precisely, the limiting proportion of balls of types $2$ is given by the Dirichlet component $D^{(a+1)}$, as all components of the vectors $v_1, \ldots, v_{a+c}$ in type $3$ colours are zero. On the other hand, Lemma \ref{conv} ensures that the proportions $\pi_k(X_n)/(rn+|X_0|)$ almost surely converge for $k
\in \{a+1, \ldots, a+c\}$, and we may rewrite the limit as a product
\begin{align*}
\frac{\Xi_k}{r} = D^{(a+1)}\Gamma_k - \frac{\pi_k(X_0)}{|X_0|}.
\end{align*}
The claimed independence of $D^{(a+1)}$ and $\Gamma_k$ is straightforward, as the urn process on classes $\mathcal{C}_{a+1}, \ldots, \mathcal{C}_d$ can be regarded as an independent urn process observed at random time steps. 
In total, this yields the representation
\begin{align} \label{V}
V= \sum_{k=1}^{a+c}\left(\frac{\Xi_k}{r} +
  \frac{\pi_k(X_0)}{|X_0|}\right)  v_k 
\end{align}
for the proportion vector $V$ via the martingale limits of Lemma
\ref{conv}. Unfortunately, due to the fact that $\Xi_{a+1}, \ldots, \Xi_{a+c}$ arise as martingale limits, their distribution is not explicit. 
\end{rem}

\begin{rem}
All random variables $D^{(1)}, \ldots,
D^{(a+1)}\Gamma_{a+1}, \ldots, D^{(a+1)}\Gamma_{a+c}$ are strictly
positive almost
surely: This is immediate for $D^{(1)}, \ldots, D^{(a+1)}$. Moreover, we have already argued that for $k
\in \{a+1, \ldots, a+c\}$,
\begin{align*}
 D^{(a+1)}\Gamma_k = \lim_{n \to \infty} \frac{\pi_k(X_n)}{rn+|X_0|} =
  \lim_{n \to \infty}\frac{|(X_n)_{\mathcal{C}_k}|}{rn+|X_0|}
\end{align*}
is the almost sure limit of the proportion of balls in class
$\mathcal{C}_k$. For any given urn model that satisfies (A$1$) to (A$4$), due to (A$3$), with probability one there is a finite time $n$ at which there is at least one ball of each dominant class in the urn.  
But as soon as there is at least one ball of each dominant class in the urn, we may compare the urn process on the type $2$ and type $3$ classes to a classical P\'{o}lya urn, where type $3$ classes are also dominant by disregarding their reinforcement of type $2$ classes: Each time a ball of colour $i$ in a type $3$ class is
drawn, instead of following the original rules, give all its children
outside the class of $i$ colour $i$, too. In the modified urn, the proportions of
colours in type $2$ classes tend to a strictly positive limit almost surely. As
there are at least as many type $2$ balls in this urn as in the
urn with unchanged colours at the same time, the claim follows.
\end{rem}

\begin{cor}[Random limits] \label{variance}
Under conditions (A$1$) to (A$4$), $\Xi_1,
\ldots, \Xi_{a+c}$ are almost surely non--degenerated unless $r$ is
simple. In this case,
$\Xi_1 = 0$. $\Xi_{a+c+1}, \ldots, \Xi_p$ are almost surely non degenerated.
\end{cor}

\begin{proof}
For the type $1$ limits $\Xi_1,
\ldots, \Xi_{a}$, the claim immediately follows from Theorem
\ref{Assi} and (\ref{rel}).

More generally and without reference to Theorem \ref{Assi}, we can use
orthogonality of martingale increments to see that for $k \in \{1, \ldots,
p\}$, 
\begin{align*}
\E\left[|\Xi_k|^2\right] &= \E\left[\left|\Xi_k - M_0^{(k)}\right|^2\right] =
                           \sum_{j=0}^{\infty}
                           \E\left[\left|M_{j+1}^{(k)} -
                           M_j^{(k)}\right|^2 \right]\\
&=\sum_{j=0}^{\infty} \left|\gamma_{j+1}^{(k)}\right|^{-2}
  \E\left[\left|\pi_k(X_{j+1}-X_j) -
  \frac{\lambda_k}{rj+|X_0|}\pi_k(X_j)\right|^2\right] .
\end{align*}
If one of the random variables $\Xi_1, \ldots, \Xi_p$ was almost surely equal to its expectation $0$, say $\Xi_k$, we had
\begin{align*}
\E\left[|\Xi_k|^2\right]=0,
\end{align*}
and thus the evolution of the urn process along projection $k$ would be completely determined by
\begin{align} \label{det}
\pi_k(X_{j+1}-X_j) =
  \frac{\lambda_k}{rj+|X_0|}\pi_k(X_j)
\end{align}
 almost surely for all $j \geq 0$. This in particular means that the
 value of $\pi_k(X_{j+1}-X_j)$ is independent of the colour of the
 $(j+1)$-th ball drawn from the urn. We will argue that this is not
 possible under our assumptions. 
 
First assume that there is an initial configuration $X_0$ that is
compatible with (A$1$) to (A$4$) and has $\pi_k(X_0)=0$. Under
this initial configuration,
$\pi_k(X_j)=0$ for all $j \ge 0$ almost surely because of (\ref{det}).
On the other hand, of $N_{j+1}$ denotes the colour drawn immediately before time $j+1$, we have
\begin{align*}
0 = \pi_k(X_{j+1}-X_j)= \pi_k\left(\Delta_{N_{j+1}} \right) = \lambda_k \left(u_k^\ast\right)^{(\Delta_{N_{j+1}})} 
\end{align*}
for
all $j \geq 0$ almost surely. Now (A$3$) ensures that for each colour $f \in \{1, \ldots, q\}$, there is $n \in \Nset_0$ with
$\Prob(X_{n}^{(f)} > 0) > 0$, which allows to conclude $u_k^{(f)} = 0$. Thus $u_k = 0$, yielding a contradiction.

The last paragraph showed that there is no admissible choice of initial configuration such that
$\pi_k(X_0) = 0$. So (\ref{det}) implies that $\pi_k(X_{j+1}-X_j) \neq
0$ for all $j$ almost surely. But due to our choice of $\{u_1, \ldots, u_q\}$, this is only possible in one particular case: Because $u_k$ is zero on all components that belong to dominant classes different from the class of $\lambda_k$, $\pi_k(X_{j+1}-X_j)=0$ each time a ball from one of these other classes is drawn. However, as the proportion of each dominant class tends to a strictly positive limit almost surely, there is a positive probability of having $\pi_k(X_{j+1}-X_j)=0$, unless there is only one dominant class and $\lambda_k$ belongs to this dominant class. In this case,
there is a time $N$ from which on there
are balls of each colour of the unique dominant class in the urn. This implies that
$u_k^{(i)} = u_k^{(j)}$ for all colours $i,j$ in this class. So $\lambda_k$ is the simple dominant eigenvalue $r$ and no other projection that can induce a deterministic limit.
\end{proof}

\begin{lem}[Speed of convergence] \label{Speed}
Let $k \in \{1, \ldots, q\}$ be such that $\Re(\lambda_k)>r/2$. Then 
\begin{align}
\left\| \frac{\Gamma(|X_0|/r + \lambda_k/r)}{\Gamma(|X_0|/r)}\Xi_k - M_n^{(k)}\right\|_{L^2} = \bo\left(n^{1/2 - \Re(\lambda_k)/r}\right)
\end{align}
as $n \to \infty$.
\end{lem}

\begin{proof}
We use the decomposition
\begin{align*}
& \left\| \frac{\Gamma(|X_0|/r + \lambda_k/r)}{\Gamma(|X_0|/r)}\Xi_k - M_n^{(k)}\right\|_{L^2}^2 = \sum_{j=n}^{\infty}
                                 \E\left[\left|M_{j+1}^{(k)} -
                                 M_j^{(k)}\right|^2\right] \\
&= \sum_{j=n}^{\infty} \left|\gamma_{j+1}^{(k)}\right|^{-2}
  \left(\E\left[|\pi_k(X_{j+1}-X_j)|^2\right] -\left|\frac{\lambda_k}{rj+|X_0|} \right|^2 \E\left[|\pi_k(X_j)|^2\right] \right)\\
  & \leq \sum_{j=n}^{\infty} \left|\gamma_{j+1}^{(k)}\right|^{-2}\E\left[\left|\pi_k(X_{j+1}-X_j)\right|^2\right] \leq C n^{1-2\Re(\lambda_k)/r}
\end{align*}
as $|\pi_k(X_{j+1}-X_j)|^2$ can only take $q$ values, independently
of $j$.
\end{proof}

\section{Proof of Theorem \ref{fluc}} \label{Sec_proof}
After the separate consideration of projections in the previous section, we now study their joint fluctuations. To this end, recall that $p=\max\{k: \Re(\lambda_k)/r > 1/2\}$ and set 
\begin{eqnarray*}
P_n:=\left(
  \begin{array}{c}
\Re(\pi_1(Y_n) - \gamma_n^{(1)} \frac{\Gamma(|X_0|/r + \lambda_1/r)}{\Gamma(|X_0|/r)}\Xi_1) \\
    \Im(\pi_1(Y_n) - \gamma_n^{(1)}\frac{\Gamma(|X_0|/r + \lambda_1/r)}{\Gamma(|X_0|/r)}\Xi_1) \\
   \Re(\pi_2(Y_n) - \gamma_n^{(2)}\frac{\Gamma(|X_0|/r + \lambda_2/r)}{\Gamma(|X_0|/r)}\Xi_2) \\
    \Im(\pi_2(Y_n) - \gamma_n^{(2)} \frac{\Gamma(|X_0|/r + \lambda_2/r)}{\Gamma(|X_0|/r)}\Xi_2) \\
      \vdots \\
    \Re(\pi_p(Y_n) - \gamma_n^{(p)} \frac{\Gamma(|X_0|/r + \lambda_p/r)}{\Gamma(|X_0|/r)}\Xi_p) \\
   \Im(\pi_p(Y_n) - \gamma_n^{(p)} \frac{\Gamma(|X_0|/r +\lambda_p/r)}{\Gamma(|X_0|/r)}\Xi_p ) \\
      \Re(\pi_{p+1}(Y_n)) \\
     \Im(\pi_{p+1}(Y_n)) \\
\vdots \\
    \Re(\pi_{q}(Y_n)) \\
     \Im(\pi_{q}(Y_n)) 
  \end{array}\right).
\end{eqnarray*}
Theorem \ref{fluc} distinguishes two cases: If there is no dominant $k$ such that $\Re(\lambda_k)/r=1/2$,
we set
\begin{align*}
Z_n := \frac{1}{\sqrt{n}} P_n.
\end{align*} 
If there is such a $k$, we normalise to
\begin{align*}
Z_n := \frac{1}{\sqrt{n \log n}} P_n.
\end{align*}
Note that components of $Z_n$ may be equal or $0$. The aim of the current section is to show the following theorem.

\begin{thm}\label{CW}
As $n \to \infty$,
\begin{align*}
Z_n \stackrel{\mathcal{L}}{\longrightarrow} \mathcal{N}(0, \Sigma_V),
\end{align*}
where $\mathcal{N}(0, \Sigma_V)$ denotes a Gaussian mixture distribution with 
latent distribution $\mathcal{L}(V)$ and covariance matrix
$\Sigma_V$ defined in (\ref{Sigma1}) to (\ref{Sigma4}) and (\ref{Sigma5}).
\end{thm}

We now explicitly give the covariance matrices. As in Theorem \ref{fluc}, there are two cases. First assume that
for all $k \in \{1, \ldots, q\}$, $\Re(\lambda_k) \not= r/2$. The
non--zero entries of the $2q \times 2q$ matrix $\Sigma_V$ are given by
\begin{align} \label{Sigma1}
&(\Sigma_V)_{2k-1,2\ell-1} := \\
&
\begin{cases}
r^2 \left(\frac{\Xi_k}{r}
  + \frac{\pi_k(X_0)}{|X_0|}\right) \left(1-\left(\frac{\Xi_k}{r}
  + \frac{\pi_k(X_0)}{|X_0|}\right)\right), \qquad k=\ell, \lambda_{k}=r
\\
-r^2 \left(\frac{\Xi_k}{r}
  + \frac{\pi_k(X_0)}{|X_0|}\right)\left(\frac{\Xi_{\ell}}{r}
  + \frac{\pi_{\ell}(X_0)}{|X_0|}\right), \qquad k\not=\ell, \lambda_{k}=\lambda_{\ell}=r
\\
\sum_{m=1}^q
  V^{(m)}\Re\left(\frac{\left(\frac{\bar{\lambda}_k+\bar{\lambda}_{\ell}}{r}-1\right)\lambda_k\lambda_{\ell}\bar{u}_k^{(m)}\bar{u}_{\ell}^{(m)}}{2\left|1-\frac{\lambda_k+\lambda_{\ell}}{r}\right|^2}
  +
  \frac{\left(\frac{\lambda_k+\bar{\lambda}_{\ell}}{r}-1\right)\bar{\lambda}_k\lambda_{\ell}u_k^{(m)}\bar{u}_{\ell}^{(m)}}{2\left|1-\frac{\lambda_k+\bar{\lambda}_{\ell}}{r}\right|^2}\right),
 a+c < k, \ell \leq p \\
\sum_{m=1}^q
  V^{(m)}\Re\left(\frac{\left(1-\frac{\bar{\lambda}_k+\bar{\lambda}_{\ell}}{r}\right)\lambda_k\lambda_{\ell}\bar{u}_k^{(m)}\bar{u}_{\ell}^{(m)}}{2\left|1-\frac{\lambda_k+\lambda_{\ell}}{r}\right|^2}
  +
  \frac{\left(1-\frac{\lambda_k+\bar{\lambda}_{\ell}}{r}\right)\bar{\lambda}_k\lambda_{\ell}u_k^{(m)}\bar{u}_{\ell}^{(m)}}{2\left|1-\frac{\lambda_k+\bar{\lambda}_{\ell}}{r}\right|^2}\right),
k, \ell >p \nonumber
\end{cases}
\end{align}
and
\begin{align}\label{Sigma2}
&(\Sigma_V)_{2k,2\ell} := \\
&
\begin{cases}
\sum_{m=1}^q
  V^{(m)}\Re\left(-\frac{\left(\frac{\bar{\lambda}_k+\bar{\lambda}_{\ell}}{r}-1\right)\lambda_k\lambda_{\ell}\bar{u}_k^{(m)}\bar{u}_{\ell}^{(m)}}{2\left|1-\frac{\lambda_k+\lambda_{\ell}}{r}\right|^2}
  +
  \frac{\left(\frac{\lambda_k+\bar{\lambda}_{\ell}}{r}-1\right)\bar{\lambda}_k\lambda_{\ell}u_k^{(m)}\bar{u}_{\ell}^{(m)}}{2\left|1-\frac{\lambda_k+\bar{\lambda}_{\ell}}{r}\right|^2}\right),
 a+c < k, \ell \leq p\\
\sum_{m=1}^q
  V^{(m)}\Re\left(-\frac{\left(1-\frac{\bar{\lambda}_k+\bar{\lambda}_{\ell}}{r}\right)\lambda_k\lambda_{\ell}\bar{u}_k^{(m)}\bar{u}_{\ell}^{(m)}}{2\left|1-\frac{\lambda_k+\lambda_{\ell}}{r}\right|^2}
  +
  \frac{\left(1-\frac{\lambda_k+\bar{\lambda}_{\ell}}{r}\right)\bar{\lambda}_k\lambda_{\ell}u_k^{(m)}\bar{u}_{\ell}^{(m)}}{2\left|1-\frac{\lambda_k+\bar{\lambda}_{\ell}}{r}\right|^2}\right),
k, \ell >p \nonumber 
\end{cases}
\end{align}
and
\begin{align}\label{Sigma3}
&(\Sigma_V)_{2k-1,2\ell} := \\
&
\begin{cases}
\sum_{m=1}^q
  V^{(m)}\Im\left(\frac{\left(\frac{\bar{\lambda}_k+\bar{\lambda}_{\ell}}{r}-1\right)\lambda_k\lambda_{\ell}\bar{u}_k^{(m)}\bar{u}_{\ell}^{(m)}}{2\left|1-\frac{\lambda_k+\lambda_{\ell}}{r}\right|^2}
  +
  \frac{\left(\frac{\lambda_k+\bar{\lambda}_{\ell}}{r}-1\right)\bar{\lambda}_k\lambda_{\ell}u_k^{(m)}\bar{u}_{\ell}^{(m)}}{2\left|1-\frac{\lambda_k+\bar{\lambda}_{\ell}}{r}\right|^2}\right),
 a+c < k, \ell \leq p\\
\sum_{m=1}^q
  V^{(m)}\Im\left(\frac{\left(1-\frac{\bar{\lambda}_k+\bar{\lambda}_{\ell}}{r}\right)\lambda_k\lambda_{\ell}\bar{u}_k^{(m)}\bar{u}_{\ell}^{(m)}}{2\left|1-\frac{\lambda_k+\lambda_{\ell}}{r}\right|^2}
  +
  \frac{\left(1-\frac{\lambda_k+\bar{\lambda}_{\ell}}{r}\right)\bar{\lambda}_k\lambda_{\ell}u_k^{(m)}\bar{u}_{\ell}^{(m)}}{2\left|1-\frac{\lambda_k+\bar{\lambda}_{\ell}}{r}\right|^2}\right),
k, \ell >p \nonumber 
\end{cases}
\end{align}
as well as 
\begin{align} \label{Sigma4}
(\Sigma_V)_{2k,2\ell-1}:=(\Sigma_V)_{2\ell-1,2k}.
\end{align}

\vspace{0.2 cm}
\noindent In the case where there is a dominant $k$ such that
$\Re(\lambda_k)/r=1/2$, the matrix $\Sigma_V$ has a lot more zero
entries due to the scaling. Its non zero entries are in places $(2k-1,
2k-1)$ and $(2k,2k)$ for $k$ such that $\Re(\lambda_k)/r=1/2$. For
these $k$,
\begin{align} \label{Sigma5}
(\Sigma_v)_{2k-1,2k-1} = (\Sigma_v)_{2k,2k} = \frac{|\lambda_k|^2}{2}\sum_{m=1}^qV^{(m)}|u_k^{(m)}|^2.
\end{align}

\vspace{0.2 cm}
{\bf Comments on the covariance structure $\Sigma_V$.} Let us first
consider the case where there is no eigenvalue with real part
$r/2$. In this case, we make the following comments on the covariance structure.
\begin{enumerate}[itemsep=0pt]
\item In the limit,
$Z_{n}^{(k)}$ and $Z_{n}^{(\ell)}$ are independent for $k \in \{1, \ldots,
2p\}$, $\ell \in \{2p+1, \ldots, 2q\}$.
\item Components
$Z_n^{(1)}, \ldots, Z_n^{(2(a+c)-1)}$ are asymptotically independent of all other
components.
\item According to our calculations,
for eigenvalues $\lambda_k$ corresponding to type $3$ classes, the
fluctuations of their projections vanish in the $\sqrt{n}$
scaling. This might be due to the fact that there are too little draws
from these classes compared to the other classes. So Theorem \ref{CW}
says nothing about the fluctuations within these classes (or, at
least, nothing particularly interesting), as the draws from the
dominant colours dominate
in the limit and there is too little fluctuation among the remaining colours.
\end{enumerate}
If there are eigenvalues with real part $r/2$:
\begin{enumerate}
\item The fluctuations
of the other projections are still of order $\sqrt{n}$ and they tend
to zero in the $\sqrt{n \log n}$ scaling.
\item Real and imaginary parts of
the $\Re(\lambda_k)=r/2$-components are independent, and they are also
independent of each other.
\end{enumerate}

\begin{rem}
The covariance matrix $\Sigma_V$ for the cyclic urn model has been calculated explicitly, see \cite{MN}.  
\end{rem}

The final ingredient for the proof of Theorem \ref{CW} is Corollary $3.1$ from
 \cite{HH}.

\begin{proposition}\label{inc}
Let $\{S_{n,j}, \mathcal{F}_{n,j}, 1\leq j \leq k_n, n \geq 1\}$ be a zero--mean,
square--integrable martingale array with increments $I_{n,j}$ and let
$\eta^2$ be an a.s. finite random variable. Suppose that for all
$\varepsilon >0$,
\begin{align}\label{bed1}
 \sum_{j=1}^{k_n}\E[I_{n,j}^2I(|I_{n,j}|>\varepsilon)|\mathcal{F}_{n,j-1}]\stackrel{\Prob}{\to} 0,
\end{align}
and
\begin{align}\label{bed2}
\sum_{j=1}^{k_n}\E[I_{n,j}^2|\mathcal{F}_{n,j-1}]\stackrel{\Prob}{\to}\eta^2,
\end{align}
and $\mathcal{F}_{n,j}\subseteq \mathcal{F}_{n+1,j}$ for $1 \leq j \leq
k_n$, $n \geq 1$. Then
\begin{align*}
S_{n,k_n} = \sum_{j=1}^{k_n}I_{n,j} \stackrel{\mathcal{L}}{\longrightarrow} \eta' N,
\end{align*}
where $\mathcal{L}(\eta')= \mathcal{L}(\eta)$, $\mathcal{L}(N) = \mathcal{N}(0,1)$ and $\eta', N$ are independent.
\end{proposition}

\begin{proof} [Proof of Theorem \ref{CW}]
Let $\alpha_1, \ldots, \alpha_{2q} \in \Rset$. We use Proposition \ref{inc} to show weak convergence of the random variables
\begin{align*}
\sum_{k=1}^{2q}\alpha_{k}Z_{n}^{(k)},
\end{align*}
from which the claim follows by an application of the Cram\'{e}r--Wold device. 

First, assume that for all eigenvalues $\lambda_k$ belonging to dominant classes, $\Re(\lambda_k)
  \not= r/2$, which is case $1$ in Theorem \ref{fluc}. In this case,
  $Z_n=P_n/\sqrt{n}$. We rewrite the given linear combination $\alpha_1Z_{n}^{(1)} + \cdots + \alpha_{2q} Z_{n}^{(2q)}$ as a sum of
martingale differences by a simultaneous consideration of the real and imaginary
parts of each eigenspace coefficient: For $1 \leq k \leq
p$, write 
\begin{align*}
&\alpha_{2k-1}Z_{n}^{(2k-1)}+ \alpha_{2k}Z_{n}^{(2k)} \\
&=
                                                        \frac{1}{\sqrt{n}}\sum_{j=n}^{\infty}(
                                                        \alpha_{2k-1}\Re(\gamma_n^{(k)}(M^{(k)}_{j}
                                                        -
                                                        M^{(k)}_{j+1}))
                                                        + \alpha_{2k}\Im(\gamma_n^{(k)}(M^{(k)}_{j}
                                                        -
                                                        M^{(k)}_{j+1})))\\
&=
  \frac{1}{\sqrt{n}}(\alpha_{2k-1}\Re(\gamma_n^{(k)})+\alpha_{2k}\Im(\gamma_n^{(k)}))
  \sum_{j=n}^{\infty}\Re(M^{(k)}_{j}
                                                        -
                                                        M^{(k)}_{j+1})\\
&+ \frac{1}{\sqrt{n}}(\alpha_{2k}\Re(\gamma_n^{(k)})-\alpha_{2k-1}\Im(\gamma_n^{(k)}) \sum_{j=n}^{\infty}\Im(M^{(k)}_{j}
                                                        -
                                                        M^{(k)}_{j+1})\\
&=: \beta_{2k-1}(n) \sum_{j=n}^{\infty} \Re(M^{(k)}_{j}
                                                        -
                                                        M^{(k)}_{j+1})+ \beta_{2k}(n) \sum_{j=n}^{\infty} \Im(M^{(k)}_{j}
                                                        -
                                                        M^{(k)}_{j+1}).
\end{align*}
In the case $p+1
\leq k \leq q$, we set $g:= \max\left\{-\frac{\lambda_k + |X_0|}{r}+1 :
1 \leq k \leq q, \lambda_k + |X_0| \in r \mathbb{Z}_{-}\right\}$. We
can only rewrite $\alpha_{2k-1}Z_{n}^{(2k-1)}+
\alpha_{2k}Z_{n}^{(2k)}$ in the same manner as above for all $k \in
\{p+1, \ldots, q\}$ if $n >g$, as the factors $\gamma_n^{(k)}$ might
be zero for smaller $n$. So for $n >g$, 
\begin{align*}
&\alpha_{2k-1}Z_{n}^{(2k-1)}+ \alpha_{2k}Z_{n}^{(2k)} \\
&= \frac{1}{\sqrt{n}}(\alpha_{2k-1}\Re(\gamma_n^{(k)})+\alpha_{2k}\Im(\gamma_n^{(k)})) \sum_{j=g}^{n-1}\Re(M^{(k)}_{j+1}
                                                        -
                                                        M^{(k)}_{j})\\
&+ \frac{1}{\sqrt{n}}(\alpha_{2k}\Re(\gamma_n^{(k)})-\alpha_{2k-1}\Im(\gamma_n^{(k)})) \sum_{j=g}^{n-1}\Im(M^{(k)}_{j+1}
                                                        -
                                                        M^{(k)}_{j})\\
& +
  \frac{1}{\sqrt{n}}(\alpha_{2k-1}\Re(\gamma_n^{(k)})+\alpha_{2k}\Im(\gamma_n^{(k)}))
  \Re (M^{(k)}_{g})+
  \frac{1}{\sqrt{n}}(\alpha_{2k}\Re(\gamma_n^{(k)})-\alpha_{2k-1}\Im(\gamma_n^{(k)}))
  \Im(M^{(k)}_{g})\\
&=: \beta_{2k-1}(n) \sum_{j=g}^{n-1}\Re(M^{(k)}_{j+1}
                                                        -
                                                        M^{(k)}_{j})+ \beta_{2k}(n) \sum_{j=g}^{n-1}\Im(M^{(k)}_{j+1}
                                                        -
                                                        M^{(k)}_{j})
  +r_k(n).
\end{align*}
With $r(n):=\sum_{k=p+1}^q r_k(n)$, for $n >g$,
\begin{align*}
\sum_{k=1}^{2q}\alpha_{k}Z_{n}^{(k)}
& = \sum_{k=1}^{p}( \beta_{2k-1}(n) \sum_{j=n}^{\infty}\Re(M^{(k)}_{j}
                                                        -
                                                        M^{(k)}_{j+1})+ \beta_{2k}(n) \sum_{j=n}^{\infty}\Im(M^{(k)}_{j}
                                                        -
                                                        M^{(k)}_{j+1}))\\
 &+ \sum_{k=p+1}^{q}( \beta_{2k-1}(n) \sum_{j=g}^{n-1}\Re(M^{(k)}_{j+1}
                                                        -
                                                        M^{(k)}_{j})+ \beta_{2k}(n) \sum_{j=g}^{n-1}\Im(M^{(k)}_{j+1}
                                                        -
                                                        M^{(k)}_{j}))
   + r(n).
\end{align*}
Next, we replace the series by a finite sum by choosing a sequence
$(k(n))_{n \geq 0} \uparrow \infty$ such that
\begin{align*}
\sum_{k=1}^{2q}\alpha_{k}Z_{n}^{(k)}
& = \sum_{k=1}^{p}( \beta_{2k-1}(n) \sum_{j=n}^{k(n)}\Re(M^{(k)}_{j}
                                                        -
                                                        M^{(k)}_{j+1})+ \beta_{2k}(n) \sum_{j=n}^{k(n)}\Im(M^{(k)}_{j}
                                                        -
                                                        M^{(k)}_{j+1}))\\
 &+ \sum_{k=p+1}^{q}( \beta_{2k-1}(n) \sum_{j=g}^{n-1}\Re(M^{(k)}_{j+1}
                                                        -
                                                        M^{(k)}_{j})+ \beta_{2k}(n) \sum_{j=g}^{n-1}\Im(M^{(k)}_{j+1}
                                                        -
                                                        M^{(k)}_{j}))
  + \varepsilon(n),
\end{align*}
where
$\varepsilon(n) \to 0$ in $L^2$. The following lemma shows that
$(k(n))_{n \ge 0} = (n^2)_{n \ge 0}$ is sufficient.

\begin{lem}\label{Lemma_rest}
Let 
\begin{align*}
\varepsilon(n):= \sum_{k=1}^{p}( \beta_{2k-1}(n) \sum_{j=n^2+1}^{\infty}\Re(M^{(k)}_{j}
                                                        -
                                                        M^{(k)}_{j+1})+ \beta_{2k}(n) \sum_{j=n^2+1}^{\infty}\Im(M^{(k)}_{j}
                                                        -
                                                        M^{(k)}_{j+1}))
  + r(n).
\end{align*}
 Then
\begin{flalign*}
&&\varepsilon(n) \stackrel{L^2}{\longrightarrow} 0, && n \to \infty.
\end{flalign*}
\end{lem}

\begin{proof}[Proof of Lemma \ref{Lemma_rest}]
It is easy to see that $r(n)$ tends to zero in $L^2$ as
$\Re(\lambda_k)<r/2$ for all summands in this term.
The remaining part follows immediately from Lemma \ref{Speed}.
\end{proof}

For each $n >g$, we have
decomposed the sum $\alpha_1Z_{n}^{(1)} + \cdots + \alpha_{2q}
Z_{n}^{(2q)}$ into a sum of weighted martingale differences with
respect to the filtration $(\mathcal{F}_{n,i})_{1 \leq i \leq n^2}$
given by
$\mathcal{F}_{n,i}:=\sigma\left(X_0, \ldots, X_{i+1} \right)$. This
yields the zero--mean, square--integrable martingale array $\{S_{n,i},
\mathcal{F}_{n,i}, g \leq i \leq n^2, n \geq g
\}$, where
\begin{align*}
 S_{n,i}&:= \sum_{k=1}^{p}( \beta_{2k-1}(n) \sum_{j=n}^{i}\Re(M^{(k)}_{j}
                                                        -
                                                        M^{(k)}_{j+1})+ \beta_{2k}(n) \sum_{j=n}^{i}\Im(M^{(k)}_{j}
                                                        -
                                                        M^{(k)}_{j+1}))\\
 &+ \sum_{k=p+1}^{q}( \beta_{2k-1}(n) \sum_{j=g}^{\min\{i,n-1\}}\Re(M^{(k)}_{j+1}
                                                        -
                                                        M^{(k)}_{j})+ \beta_{2k}(n) \sum_{j=g}^{\min\{i,n-1\}}\Im(M^{(k)}_{j+1}
                                                        -
                                                        M^{(k)}_{j})).
\end{align*}
Weak
convergence of this array implies weak convergence of $\alpha_1Z_{n}^{(1)} + \cdots + \alpha_{2q}
Z_{n}^{(2q)}$, and it remains to
check the conditions of Proposition \ref{inc}.  

As
$\mathcal{F}_{n,i}$ is independent of $n$, the
filtration satisfies the condition in Theorem
\ref{inc}.
Depending on the summation index $j$, the increments $I_{n,j}$ range
over different projections. We use the shorthand
\begin{align*}
I_{n,j}:=
\begin{cases}
\sum_{k=p+1}^{q}( \beta_{2k-1}(n) \Re(M^{(k)}_{j+1}
                                                        -
                                                        M^{(k)}_{j})+ \beta_{2k}(n) \Im(M^{(k)}_{j+1}
                                                        -
                                                        M^{(k)}_{j})),
                                                      & \hspace{0.3
                                                        cm} j < n,\\
\sum_{k=1}^{p}( \beta_{2k-1}(n) \Re(M^{(k)}_{j}
                                                        -
                                                        M^{(k)}_{j+1})+ \beta_{2k}(n) \Im(M^{(k)}_{j}
                                                        -
                                                        M^{(k)}_{j+1})),
                                                     & \hspace{0.3 cm}
                                                      j \geq n.
\end{cases}
\end{align*} 
The absolute value of these increments is deterministically bounded:
For $j < n$,
\begin{align*}
|I_{n,j}|& \leq \sum_{k=p+1}^{q}| \beta_{2k-1}(n)| |\Re(M^{(k)}_{j+1}
                                                        -
                                                        M^{(k)}_{j})|+ |\beta_{2k}(n)| |\Im(M^{(k)}_{j+1}
                                                        -
                                                        M^{(k)}_{j})|\\
& \leq C_1 \sum_{k= p+1}^q n^{\Re(\lambda_{k})/r-1/2 }(|\Re(M^{(k)}_{j+1}
                                                        -
                                                        M^{(k)}_{j})|+|\Im(M^{(k)}_{j+1}
                                                        -
                                                        M^{(k)}_{j})|)\\
& \leq\sqrt{2} C_1 \sum_{k=p+ 1}^q n^{\Re(\lambda_{k})/r-1/2 }|M^{(k)}_{j+1}
                                                        -
                                                        M^{(k)}_{j}|
 \leq C_2  n^{-1/2}\sum_{k=p+1}^q
  \left(\frac{n}{j}\right)^{\Re(\lambda_{k})/r}\\
& = \bo\left(n^{\max\{\Re(\lambda_{p+1})/r,0\}-1/2} \right)
\end{align*}
as $n \to \infty$, where $C_1$ and $C_2$ are
positive constants. Analoguously, for $n \leq j \leq n^2$, 
\begin{align*}
|I_{n,j}|
& \leq C  n^{-1/2}\sum_{k=1}^p
  \left(\frac{n}{j}\right)^{\Re(\lambda_{k})/r} = \bo\left(n^{-1/2}\right)
\end{align*}
as $n \to \infty$, where $C>0$ is a
constant.

By the above, for each $\varepsilon >0$, there exists $N \in
\mathbb{N}$ such that $|I_{N,j}|<\varepsilon$ for all $j=g,
\ldots, N^2$, and in particular
\begin{align*}
\sum_{j=g}^{n^2}\E[I_{n,j}^2 \mathrm{1}(|I_{n,j}|>
\varepsilon)|\mathcal{F}_{n,j-1}] = 0
\end{align*}
for all $n \geq N$. Thus, the given martingale array satisfies condition (\ref{bed1}). 

We now turn to condition (\ref{bed2}). For $z \in \mathbb{C}^q$, let
$\Re(z), \Im(z)$ denote the vectors whose components are given by the
real (respectively imaginary) parts of the components of $z$. We
rewrite the increments $I_{n,j}$ as 
\begin{align*}
I_{n,j}^2 = \left(\xi_{n,j} (X_{j+1} - X_{j}) - \eta_{n,j} \frac{X_{j}}{rj+|X_0|}\right)^2
\end{align*}
with
\begin{align*}
\xi_{n,j}:=\sum_{k \in K}&
\frac{1}{\sqrt{n}}\left[\left(\alpha_{2k-1}\Re\left(\gamma_n^{(k)}/\gamma_{j+1}^{(k)}\right)+\alpha_{2k}\Im\left(\gamma_n^{(k)}
                           / \gamma_{j+1}^{(k)}\right)\right)\Re(u_k^\ast)
                           \right. \\
& +
  \left.  \left(\alpha_{2k}\Re\left(\gamma_n^{(k)}/\gamma_{j+1}^{(k)}\right)
  -
  \alpha_{2k-1}\Im\left(\gamma_n^{(k)}/ \gamma_{j+1}^{(k)}\right)\right)\Im(u_k^\ast)\right] 
\end{align*}
and
\begin{align*}
\eta_{n,j}:= \sum_{k \in K}&
\frac{1}{\sqrt{n}}\left[\left(\alpha_{2k-1}\Re\left(\gamma_n^{(k)}/\gamma_{j+1}^{(k)}\right)+\alpha_{2k}\Im\left(\gamma_n^{(k)}/\gamma_{j+1}^{(k)}\right)\right)\Re(\lambda_k u_k^\ast)
                           \right. \\
& +
  \left.  \left(\alpha_{2k}\Re\left(\gamma_n^{(k)}/\gamma_{j+1}^{(k)}\right)
  -
  \alpha_{2k-1}\Im\left(\gamma_n^{(k)}/\gamma_{j+1}^{(k)}\right)\right)\Im(\lambda_k u_k^\ast)\right], 
\end{align*}
where $K=\{1, \ldots, p\}$ or $K=\{p+1,
\ldots, q\}$, depending on $j$. With this,
\begin{align} \label{Wachstum}
\sum_{j=g}^{n^2}\E[I_{n,j}^2 |\mathcal{F}_{n,j-1}]= 
\sum_{j=g}^{n^2}\sum_{m=1}^q \frac{X_{j}^{(m)}}{rj+|X_0|}
                                           \left(\xi_{n,j} \Delta_{m}
                                        - \eta_{n,j}
                                           \frac{X_{j}}{rj+|X_0|}\right)^2.
\end{align}
This sum over the conditional squared increments converges
almost surely: Recall that each of the $\xi_{n,j}$ and
$\eta_{n,j}$ itself is a sum over different eigenspace components.
If $j \le n-1$, the inner sum ranges over small eigenspaces 
$k$ and $\ell$ with $k, \ell \geq p+1$ 
and therefore, almost surely,
\begin{align*}
&\sum_{j=g}^{n-1}\E[I_{n,j}^2
  |\mathcal{F}_{n,j-1}] \sim \frac{1}{n}\sum_{j=g}^{n-1}\sum_{m=1}^q
  \sum_{k, \ell=p+1}^q
  \frac{X_{j}^{(m)}}{rj+|X_0|}\left(\frac{n}{j}\right)^{\frac{\Re(\lambda_k+\lambda_{\ell})}{r}}
  \\
&\cdot \left(
  \left(\alpha_{2k-1}\Re\left(\lambda_k(u_{k}^\ast)^{(m)} -
  \frac{\lambda_k u_k^\ast X_{j}}{rj+|X_0|}\right)
   + \alpha_{2k}\Im\left(\lambda_k(u_{k}^\ast)^{(m)} - \frac{\lambda_k
  u_k^\ast X_{j}}{rj+|X_0|}\right)\right)\cos\left(
  \frac{\Im(\lambda_k)}{r}\log \left(\frac{n}{j}\right)\right)\right. \\
& +\hspace{-0.1 cm}  \left. \left(\alpha_{2k}\Re\left(\lambda_k(u_{k}^\ast)^{(m)} -
  \frac{\lambda_ku_k^\ast X_{j}}{rj+|X_0|}\right)-
  \alpha_{2k-1}\Im\left(\lambda_k(u_{k}^\ast)^{(m)} - \frac{\lambda_k
  u_k^\ast X_{j}}{rj+|X_0|}\right) \right)\sin\left(
  \frac{\Im(\lambda_k)}{r}\log\left(\frac{n}{j}\right)\right)\right)\\
&\cdot \left(
  \left(\alpha_{2\ell-1}\Re\left(\lambda_{\ell}(u_{\ell}^\ast)^{(m)} -
  \frac{\lambda_\ell u_{\ell}^\ast  X_{j}}{rj+|X_0|}\right)
   + \alpha_{2\ell}\Im\left(\lambda_{\ell}(u_{\ell}^\ast)^{(m)} -
  \frac{\lambda_\ell u_{\ell}^\ast X_{j}}{rj+|X_0|}\right)\right)\cos\left(
  \frac{\Im(\lambda_{\ell})}{r}\log\left(\frac{n}{j}\right)\right)\right. \\
& +\hspace{-0.1 cm}
  \left. \left(\alpha_{2\ell}\Re\left(\lambda_{\ell}(u_{\ell}^\ast)^{(m)}
  - \frac{\lambda_\ell u_{\ell}^\ast X_{j}}{rj+|X_0|}\right)-
  \alpha_{2\ell-1}\Im\left(\lambda_{\ell}(u_{\ell}^\ast)^{(m)} -
  \frac{\lambda_\ell u_{\ell}^\ast X_{j}}{rj+|X_0|}\right) \right)\sin\left(
  \frac{\Im(\lambda_{\ell})}{r}\log\left(\frac{n}{j}\right)\right)\right).
\end{align*}
Now, by Theorem \ref{Assi} and its extension, $X_{j}/(rj+|X_0|)$ converges to $V$ almost surely. This implies that for $\lambda_k \not= r$, $u_k^\ast(X_{j}/(rj+|X_0|)) \to 0$ almost surely and therefore, as $n \to \infty$,
\begin{align*}
&\sum_{j=g}^{n-1}\E[I_{n,j}^2
  |\mathcal{F}_{n,j-1}] \sim \frac{1}{2n}\sum_{j=g}^{n-1}\sum_{m=1}^q\sum_{k, \ell = p+1}^q
  V^{(m)}\left(\frac{n}{j}\right)^{\frac{\Re(\lambda_k+\lambda_{\ell})}{r}}\\
&
 \cdot \left(
  ((\alpha_{2k-1}\alpha_{2\ell-1}-\alpha_{2k}\alpha_{2\ell})\Re(\lambda_k\lambda_{\ell}\bar{u}_k^{(m)}\bar{u}_{\ell}^{(m)})+(\alpha_{2k-1}\alpha_{2\ell}+\alpha_{2k}\alpha_{2\ell-1})\Im(\lambda_k\lambda_{\ell}\bar{u}_k^{(m)}\bar{u}_{\ell}^{(m)}))\right.\\
& \left. \cdot \cos(
  \Im((\lambda_k+\lambda_{\ell})/r) \log(n/j))
  +
  ((\alpha_{2k-1}\alpha_{2\ell-1}+\alpha_{2k}\alpha_{2\ell})\Re(\bar{\lambda}_k\lambda_{\ell}u_k^{(m)}\bar{u}_{\ell}^{(m)}) \right. \\
& \left.+(\alpha_{2k-1}\alpha_{2\ell}-\alpha_{2k}\alpha_{2\ell-1})\Im(\bar{\lambda}_k\lambda_{\ell}u_k^{(m)}\bar{u}_{\ell}^{(m)})) \cos(
  \Im((\lambda_k-\lambda_{\ell})/r)\log(n/j))
  \right. \\
& \left. +
  ((\alpha_{2k-1}\alpha_{2\ell}+\alpha_{2k}\alpha_{2\ell-1})\Re(\lambda_k\lambda_{\ell}\bar{u}_k^{(m)}\bar{u}_{\ell}^{(m)})+(\alpha_{2k}\alpha_{2\ell}-\alpha_{2k-1}\alpha_{2\ell-1})\Im(\lambda_k\lambda_{\ell}\bar{u}_k^{(m)}\bar{u}_{\ell}^{(m)}))\right.\\
& \left. \cdot \sin(
  \Im((\lambda_k+\lambda_{\ell})/r)\log(n/j))
  +
  ((\alpha_{2k}\alpha_{2\ell-1}-\alpha_{2k-1}\alpha_{2\ell})\Re(\bar{\lambda}_k\lambda_{\ell}u_k^{(m)}\bar{u}_{\ell}^{(m)}) \right. \\
& \left.+(\alpha_{2k-1}\alpha_{2\ell-1}+\alpha_{2k}\alpha_{2\ell})\Im(\bar{\lambda}_k\lambda_{\ell}u_k^{(m)}\bar{u}_{\ell}^{(m)})) \sin(
 \Im((\lambda_k-\lambda_{\ell})/r)\log(n/j))
  \right) \\
& \longrightarrow \sum_{k, \ell = p+1}^q (\alpha_{2k-1}\alpha_{2\ell-1} (\Sigma_V)_{2k-1,2\ell-1} +
  \alpha_{2k}\alpha_{2\ell} (\Sigma_V)_{2k,2\ell} +\alpha_{2k-1}\alpha_{2\ell} (\Sigma_V)_{2k-1,2\ell}\\
  & \qquad \qquad \qquad + \alpha_{2k}\alpha_{2\ell-1} (\Sigma_V)_{2k,2\ell-1}).
\end{align*}
For the remaining increments $\E[I_{n,j}^2
  |\mathcal{F}_{n,j-1}]$, $j=n, \ldots, n^2$, where the inner sum ranges over large eigenspaces, we have to distinguish three cases. A calculation analoguous to the previous one shows the claimed convergence of summands $k, \ell$
with $1/2 < \Re(\lambda_k)/r, \Re(\lambda_{\ell})/r <1$. 
Next, if $1/2 < \Re(\lambda_k)/r <1$ and
$\lambda_{\ell}=r$, due to our choice of bases $u_1, \ldots, u_q$ and $v_1, \ldots, v_q$, the corresponding summand converges to
\begin{align*}
&r \alpha_{2\ell-1} \sum_{m=1}^q
  V^{(m)}\left(u_{\ell}^{(m)}-\frac{\Xi_\ell}{r}-\pi_\ell(X_0)/|X_0|\right)\left(\alpha_{2k-1}\Re\left(\frac{\lambda_k \bar{u}_k^{(m)}}{2+\lambda_k/r} \right) + \alpha_{2k}\Im\left(\frac{\lambda_k \bar{u}_k^{(m)}}{2+\lambda_k/r}\right)\right)=0. 
\end{align*}
Finally, for summands $k, \ell$ with $\lambda_k = \lambda_{\ell} = r$, the corresponding summand tends to
\begin{align*}
& r^2 \alpha_{2k-1}\alpha_{2\ell-1} \sum_{m=1}^qV^{(m)}\left(u_k^{(m)} - \left(\frac{\Xi_k}{r}
  + \frac{\pi_k(X_0)}{|X_0|}\right)\right)\left(u_{\ell}^{(m)} - \left(\frac{\Xi_{\ell}}{r}
  + \frac{\pi_{\ell}(X_0)}{|X_0|}\right)\right)\\
= &\begin{cases}
-r^2 \alpha_{2k-1}\alpha_{2\ell-1} \left(\frac{\Xi_k}{r}
  + \frac{\pi_k(X_0)}{|X_0|}\right)\left(\frac{\Xi_{\ell}}{r}
  + \frac{\pi_{\ell}(X_0)}{|X_0|}\right), \qquad k \not= \ell \\
r^2 \alpha_{2k-1}^2\left(\frac{\Xi_k}{r}
  + \frac{\pi_k(X_0)}{|X_0|}\right) \left(1-\left(\frac{\Xi_k}{r}
  + \frac{\pi_k(X_0)}{|X_0|}\right)\right), \qquad k = \ell
\end{cases}\\
= & \quad \alpha_{2k-1}\alpha_{2\ell-1} (\Sigma_V)_{2k-1,2\ell-1}.
\end{align*}
In total, this implies that
\begin{align*}
&\sum_{j=g}^{n^2}\E[I_{n,j}^2 |\mathcal{F}_{n,j-1}]
  \stackrel{a.s.}{\longrightarrow} \sum_{i,j=1}^{2q}
  \alpha_{i}\alpha_{j} (\Sigma_V)_{i,j} = (\alpha_1, \ldots, \alpha_{2q}) \Sigma_V (\alpha_1, \ldots, \alpha_{2q})^t.
\end{align*}

Thus by Proposition \ref{inc}, 
\begin{align*}
\alpha_1Z_n^{(1)} + \cdots +
\alpha_{2q}Z_n^{(2q)} \stackrel{\mathcal{L}}{\longrightarrow} (\alpha_1, \ldots, \alpha_{2q})\mathcal{N}(0, \Sigma_V),
\end{align*}
and the Cram\'{e}r--Wold
device implies weak convergence of $(Z_n)_{n \geq 1}$ to a Gaussian
distribution with covariance matrix $\Sigma_V$ given $V$. 

In the case where there is at least one dominant $k$ such that
$\Re(\lambda_k)/r=1/2$, proceeding along the same lines as for the
first case, one can again show that 
\begin{align*}
\alpha_1Z_n^{(1)} + \cdots +
\alpha_{2q}Z_n^{(2q)} \stackrel{\mathcal{L}}{\longrightarrow} (\alpha_1, \ldots, \alpha_{2q})\mathcal{N}(0, \Sigma_V),
\end{align*}
where $\Sigma_V$ is defined in equation (\ref{Sigma5}). The only
difference is that in this case, due to
the scaling, the matrix $\Sigma_V$ has a lot more zero
entries.
\end{proof}

\begin{proof}[Proof of Theorem \ref{fluc}]
We now analyse the random covariance matrix $A_V$ and show that it is
of the given form in both cases of
Theorem \ref{fluc}. 
Due to (B$4$), we have
\begin{align*}
\frac{1}{\sqrt{n \ell_n}}
\left(Y_n - \sum_{k=1}^p
  n^{\frac{\lambda_k}{r}}\Xi_kv_k \right)\sim
  \sum_{k=1}^q
  \left(Z_{n}^{(2k-1)}\Re\left(v_k\right)-Z_{n}^{(2k)}\Im(v_k)\right)
  = M Z_n
\end{align*}
almost surely, where $\ell_n = 1$ in the first case and $\ell_n = \log n$ in the
second case. With this, Theorem \ref{CW} and the continuous mapping theorem imply that 
\begin{align*}
\frac{1}{\sqrt{n \ell_n}}
\left(Y_n - \sum_{k=1}^p
  n^{\frac{\lambda_k}{r}}\Xi_kv_k \right)
  \stackrel{\mathcal{L}}{\longrightarrow} M \mathcal{N}(0, \Sigma_V)
\end{align*}
and thus $A_V = M \Sigma_V
M^t$ in both cases.

It remains to decide in which cases $(A_V)_{j,j}>0$. First, assume that we are in case $1$ of Theorem \ref{fluc}. We do not work
with the matrix $A_V$ directly, but rather use the fact that the conditional squared increments of the approximate
martingale difference sum
\begin{align} \label{colj}
\sum_{k=1}^q \left(Z_{n}^{(2k-1)}\Re\left(v_{k}^{(j)}\right)-Z_{n}^{(2k)}\Im\left(v_{k}^{(j)}\right)\right),
\end{align}
$j=1, \ldots, q$, converge to $(A_V)_{j,j}0$ almost surely, as in the proof of Theorem \ref{CW}. In the following, let 
\begin{align*}
\alpha_1= \Re(v_1^{(j)}), \alpha_2= -\Im(v_1^{(j)}), \ldots, \alpha_{2q-1}= \Re(v_q^{(j)}), \alpha_{2q}= -\Im(v_q^{(j)}).
\end{align*}

{\em Non--dominant colours.} If $j$ belongs to a
non--dominant colour class,
$v_k^{(j)} = 0$ for all dominant colours $k$ by our choice of right
eigenvectors, and (\ref{colj}) reduces
to a sum over type $3$ colours. Now the proof of
Theorem \ref{CW} implies that the remaining summands in (\ref{colj})
converge weakly to a mixed Gaussian distribution with variance
$(A_V)_{j,j}= 0$, as all variances and covariances of
type $3$ projections are zero in the limit. 

{\em Dominant colours.} Suppose that $j$ is a dominant colour in the sense that it belongs to one of the classes $\mathcal{C}_1, \ldots, \mathcal{C}_{a+c}$, say $j \in \mathcal{C}_m$. Again, by our choice of eigenvectors,
the sum (\ref{colj}) reduces to a sum over colours in $\mathcal{C}_m$, as
$v_k^{(j)} \not= 0$ only if $k$ is a colour in class $\mathcal{C}_m$. We distinguish two subcases: First, if $a+c>1$, then the variance among the supercolours is enough to guarantee positive variance of colour $j$. More precisely, the almost sure limit of 
\begin{align*}
\sum_{j=n}^{n^2}\E[I_{n,j}^2 |\mathcal{F}_{n,j-1}]
\end{align*}
with an appropriate choice of coefficients yields an almost sure lower bound on the variance of colour $j$. Let $k_1, \ldots, k_\ell$ be the indices of large eigenvectors $v_{k_1}, \ldots, v_{k_\ell}$ associated to class $\mathcal{C}_m$, where the vector $v_{k_1}$ is associated to the eigenvalue $r$. Then for $j=n, \ldots, n^2$, omitting the $n$ in $\beta_{2k-1}(n), \beta_{2k}(n)$,
\begin{align*}
&\E[I_{n,j}^2 |\mathcal{F}_{n,j-1}] \geq  (\E[\beta_{2k_1-1}^2(M_j^{(k_1)} - M_{j+1}^{(k_1)})^2|\mathcal{F}_{n,j-1}] \\
&+ \sum_{s=2}^\ell \E[\beta_{2k_1-1}(M_j^{(k_1)} - M_{j+1}^{(k_1)})(\beta_{2k_s-1}\Re(M_j^{(k_s)} - M_{j+1}^{(k_s)})+\beta_{2k_s}\Im(M_j^{(k_s)} - M_{j+1}^{(k_s)}))|\mathcal{F}_{n,j-1}]) \\
& \longrightarrow r^2 (v_{k_1}^{(j)})^2\left(\frac{\Xi_{k_1}}{r}
  + \frac{\pi_k(X_0)}{|X_0|}\right) \left(1-\left(\frac{\Xi_{k_1}}{r}
  + \frac{\pi_k(X_0)}{|X_0|}\right)\right) >0
\end{align*}
almost surely by Corollary \ref{variance}.

If $a+c=1$, we have to follow a different route. 
As before, each of the $n^2+1-g$ summands in  (\ref{Wachstum}) is
non--negative and thus any sum over less terms yields a lower bound on
the whole sum. Only considering part of the sum has the advantage
that the variances and covariances of the fluctuations in the sum
grow at a different speed.
For example, for $n$ large and
$\varepsilon \in (0,1)$, we can either sum from $g$ to $\varepsilon n$
or from $\varepsilon^{-1} n$ to $n^2$ to get a lower bound.
In the first case, the squared increments comprise of summands with $\Re(\lambda_k),
\Re(\lambda_{\ell}) \le r/2$. A calculation
as in the proof of Theorem \ref{CW} shows that the contribution coming
from the fluctuations in projections $\pi_k, \pi_{\ell}$ to the
sum (\ref{Wachstum}) cut off at $\varepsilon n$ with coefficients as above is at most
of order $\varepsilon^{1- \Re(\lambda_k+\lambda_{\ell})/r}$. In the
second case, the squared increments comprise of summands with $\Re(\lambda_k),
\Re(\lambda_{\ell}) > r/2$. The contribution coming
from the fluctuations in projections $\pi_k, \pi_{\ell}$ to the
sum (\ref{Wachstum}) without the first $\varepsilon^{-1} n$ summands with coefficients chosen appropriately is at most
of order $\varepsilon^{\Re(\lambda_k+\lambda_{\ell})/r -1}$. In
particular, the variance contribution from projections with real part
close to $r/2$ (and nonzero coefficients) is the greatest.

We thus choose $k$ such that among all
possible eigenvalues $\lambda_k \not=0$ associated to $\mathcal{C}_m= \mathcal{C}_1$, the distance $|\Re(\lambda_k)/r- 1/2|$ is
minimal and $|v_k^{(j)}| > 0$. If there are a large eigenvalue $\lambda_{\ell}$ and a small eigenvalue $\lambda_k \not= 0$ such that
$|\Re(\lambda_k)/r- 1/2|= |\Re(\lambda_{\ell})/r- 1/2|$ is minimal, choose any of them.
Moreover, this is possible as the Perron--Frobenius
eigenvalue associated with $\mathcal{C}_1$ satisfies these conditions,
for example. Also due to assumption (A$4$), $\lambda_k$ is simple in $\mathcal{C}_1$. 

In case $1$ of Theorem \ref{fluc}, either
$\Re(\lambda_k)/r <1/2$ or $\Re(\lambda_k)/r > 1/2$. If
$\Re(\lambda_k)/r <1/2$, we choose $\varepsilon$ small enough such that
$2 \Im(\lambda_k)/r \log \varepsilon$ is a negative multiple of $2
\pi$ and cut off at $\varepsilon n$.
The dominant term now is of order
$\varepsilon^{1-2\Re(\lambda_k)/r}$. It has coefficient
\begin{align*}
\frac{1}{2}\frac{|\lambda_k|^2|v_k^{(j)}|^2}{\Re(1-2\lambda_k/r)}\sum_{\ell=1}^q
  V^{(\ell)}|u_k^{(\ell)}|^2 + \frac{1}{2|1-2 \lambda_k/r|^2} \sum_{\ell=1}^q V^{(\ell)} Re((1-2\lambda_k/r)\lambda_k^2(\bar{u}_k^{(\ell)})^2(\bar{v}_k^{(j)})^2),
\end{align*}
which is
positive (recall that $|v_k^{(j)}| > 0$ and $(\Sigma_V)_{2k-1,2k-1}>0$ or
$(\Sigma_V)_{2k,2k}>0$).

If
$\Re(\lambda_k)/r >1/2$, we choose $\varepsilon$ small enough such that
$2 \Im(\lambda_k)/r \log \varepsilon$ is a negative multiple of $2
\pi$ and start the sum at $\varepsilon^{-1} n$.
The dominant term now is of order
$\varepsilon^{2\Re(\lambda_k)/r-1}$ and has non--zero coefficient.


We finally turn to case $2$ of Theorem \ref{fluc}, in which there is at least one dominant class which has an eigenvalue $\lambda_k$ with $\Re(\lambda_k)/r =1/2$. In this case, the claim simply follows from the diagonal structure of the matrix $\Sigma_V$ and our choice of eigenvectors. 
\end{proof}

\end{document}